\documentclass[11pt]{article}
\usepackage{amsmath}
\usepackage{amssymb}
\usepackage{latexsym}
\usepackage{graphicx}

\title
{Asymptotic Euler-Maclaurin formula over lattice polytopes}
\author
{
Tatsuya Tate\thanks{Research partially supported by 
JSPS Grant-in-Aid for Scientific Research (No. 21740117).}\\
Graduate School of Mathematics \\
Nagoya University \\
Furo-cho, Chikusa-ku, \\
Nagoya, 464--8602, Japan\\
Email: tate@math.nagoya-u.ac.jp}
\date{\empty}

\newcommand{\supp}{{\operatorname{supp\,}}}

\newcommand{\ispa}[1]{\langle \,#1 \,\rangle } 
 
\newcommand{\comp}{{\scriptstyle \circ}}

\newcommand{\ol}{\overline}

\newcommand{\mf}{\mathfrak}
\newcommand{\mb}{\mathbb}

\newcommand{\ccal}{\mathcal{C}}
\newcommand{\dcal}{\mathcal{D}}
\newcommand{\ecal}{\mathcal{E}}
\newcommand{\fcal}{\mathcal{F}}

\newcommand{\hcal}{\mathcal{H}}

\newcommand{\scal}{\mathcal{S}}

\newcommand{\vcal}{\mathcal{V}}

\newcommand{\re}{{\rm Re}\,}

\newtheorem{theorem}{{\bf Theorem}}[section]
\newtheorem{maintheo}{{\bf Theorem}}

\newtheorem{cor}[theorem]{{\bf Corollary}}
\newtheorem{lem}[theorem]{{\bf Lemma}}
\newtheorem{prop}[theorem]{{\bf Proposition}}
\newenvironment{example}{\medskip\noindent{\it Example:\/} }{\medskip}
\newenvironment{rem}{\medskip\noindent{\it Remark:\/} }{\medskip}
\newtheorem{defin}[theorem]{{\bf Definition}}

\numberwithin{equation}{section}

\newenvironment{proof}%
{\def\psymbol{{\it Proof. }\enspace}
\noindent{\psymbol}}%
{\def\qed{$\square$}
\hfill\qed\par\bigskip}

\begin{document}

\maketitle

\renewcommand{\labelenumi}{{\rm (\arabic{enumi})}}
\setcounter{section}{-1}

\begin{abstract}
Formulas for the Riemann sums over lattice polytopes determined by the lattice points in the polytopes 
are often called Euler-Maclaurin formulas. 
An asymptotic Euler-Maclaurin formula, by which we mean an asymptotic expansion formula 
for Riemann sums over lattice polytopes, was first obtained by Guillemin-Sternberg \cite{GS}. 
Then, the problem is to find a concrete formula for each term of the expansion. 
In this paper, an asymptotic Euler-Maclaurin formula of the Riemann sums over general lattice polytopes is given. 
The formula given here is an asymptotic form of the so-called local Euler-Maclaurin formula of Berline-Vergne \cite{BeV}. 
For Delzant polytopes, our proof given here is independent of the local Euler-Maclaurin formula. 
Furthermore, a concrete description of differential operators which appear in each term of the 
asymptotic expansion for Delzant lattice polytopes is given. 
By using this description, when the polytopes are Delzant lattice, a concrete formula for each term of the expansion in two dimension 
and a formula for the third term of the expansion in arbitrary dimension are given. 
\end{abstract}

\section{Introduction}
\label{INTRO}

In this paper, we consider asymptotic behavior of the Riemann sums over lattice polytopes, 
\begin{equation}
\label{RS0}
R_{N}(P;\varphi):=\frac{1}{N^{\dim(P)}}\sum_{\gamma \in (NP) \cap \mb{Z}^{m}}\varphi(\gamma/N), 
\end{equation}
where $P$ is a lattice polytope in $\mb{R}^{m}$, which means that each vertex has integer coordinates, 
and $\varphi$ is a smooth function on $P$. 
Formulas for $R_{N}(P;\varphi)$, which are often called Euler-Maclaurin formulas, 
are extensively investigated in combinatorics and geometry of toric varieties. 
If we take $\varphi=1$, the Riemann sum $R_{N}(P;1)$ is reduced to 
the so-called Ehrhart polynomial
\[
E_{P}(N):=\sharp (NP) \cap \mb{Z}^{m}=N^{\dim(P)}R_{N}(P;1), 
\]
which is closely related to the Todd class of a toric variety corresponding to the polytope $P$. 
In this context, geometry of toric varieties is a suitable and powerful tool to analyze the function $E_{P}(N)$. 
Indeed, as in \cite{F}, one can show that $E_{P}(N)$ is a polynomial in $N$ by using the Hirzebruch-Riemann-Roch theorem. 
The problems concerning (exact) Euler-Maclaurin formulas and Ehrhart polynomials are investigated by various authors, 
for example \cite{CS}, \cite{BeV}, \cite{BrV}, \cite{KP}. 
See \cite{KSW2} and references therein for various results on these topics. 

Before explaining some of the results closely related to the present paper, we state one of our theorems. 
\begin{maintheo}
\label{lAEMT}
Let $P$ be a lattice polytope in $\mb{R}^{m}$.
For each face $f$ of $P$ and non-negative integer $n$ with $\dim (f) \geq \dim(P)-n$, 
there exists a homogeneous differential operator $D_{n}(P;f)$ of order $n-\dim(P)+\dim(f)$ with rational constant coefficients 
which involves derivatives only in directions orthogonal to the face $f$ such that for each 
smooth function $\varphi$ on $P$, we have the following asymptotic Euler-Maclaurin formula$:$
\begin{equation}
\label{lAEM}
R_{N}(P;\varphi) \sim \sum_{n \geq 0}N^{-n}
\sum_{f \in \fcal(P),\,\dim(f) \geq \dim(P)-n}\int_{f}D_{n}(P;f)\varphi \quad (N \to \infty), 
\end{equation}
where $\fcal(P)$ denotes the set of faces of $P$. The integration in the right hand side is performed with respect to the measure 
on the affine hull $\ispa{f}$ of $f$ which is the parallel translation of the Lebesgue measure 
on the subspace $L(f)$ parallel to $\ispa{f}$ defined by the lattice $L(f) \cap \Lambda$. 
\end{maintheo}
In this section, we explain some of the previous works on the Euler-Maclaurin formula closely 
related to Theorem \ref{lAEMT} and mention other results obtained in the present paper.

An exact Euler-Maclaurin formula for Delzant polytopes was originally obtained by 
Khovanskii-Pukhlikov \cite{KP}, and Brion-Vergne \cite{BrV} generalized it to simple polytopes 
without using the theory of toric varieties. 
One of their results can be stated as (assuming that $P$ is a Delzant polytope)
\begin{equation}
\label{eEM}
R_{N}(P;\varphi)={\rm Todd}(P;\partial/N\partial h)
\left.\int_{P_{h}}\varphi(x)\,dx\right|_{h=0}, 
\end{equation}
where $\varphi$ is a polynomial, $h=(h_{1},\ldots,h_{d}) \in \mb{R}^{d}$ is a small parameter 
with $d$ the number of faces of $P$ of codimension one, 
${\rm Todd}(z)=\frac{z}{1-e^{-z}}$ is an analytic function around the origin, called the Todd function, 
\[
{\rm Todd}(P;\partial/N\partial h)=\prod_{i=1}^{d}{\rm Todd}(\partial/N\partial h_{i})
\]
is a differential operator (of infinite order), and when the polytope $P$ is given by $P=\{x\,;\,\ispa{u_{i},x} \geq c_{i},i=1,\ldots,d\}$, 
then $P_{h}=\{x\,;\,\ispa{u_{i},x} \geq c_{i}-h_{i},i=1,\ldots,d\}$. 
Note that Brion-Vergne \cite{BrV} obtained the same formula for simple polytopes 
with a modification of the differential operator ${\rm Todd}(P;\partial/N\partial h)$.

In \cite{BeV}, Berline-Vergne obtained an effective formula for $R_{N}(P;\varphi)$ (still $\varphi$ being assumed to be polynomial), 
which they call a local Euler-Maclaurin formula. This formula is of the form (setting $N=1$ for simplicity)
\begin{equation}
\label{eLEM}
R_{1}(P;\varphi)=\sum_{f}\int_{f}D(P,f)\varphi,
\end{equation}
where the sum runs over all faces $f$ of $P$, $D(P,f)$ is a differential operator (of infinite order) with rational constant 
coefficients on $\mb{R}^{m}$ which involves derivatives only in directions perpendicular to the face $f$. 
One of remarkable points is that the formula \eqref{eLEM} of Berline-Vergne holds for any rational polytopes, 
which means that each vertex of the polytope has rational coordinates. 
They constructed a meromorphic function $\mu(\mf{a})$ for any affine rational polyhedral cone $\mf{a}$
and use a sort of inclusion-exclusion property (which is called a valuation property) of $\mu$ to 
show that it is analytic near the origin, and they define the symbol of the operator $D(P,f)$ by using $\mu$. 

The operators $D_{n}(P;f)$ in our formula \eqref{lAEM} is, by definition, the homogeneous parts of the operator $D(P,f)$ in \eqref{eLEM}. 
Thus, one can think the formula \eqref{lAEM} as an asymptotic form of the local Euler-Maclaurin formula \eqref{eLEM} due to Berline-Vergne. 
As we point out in Subsection \ref{Heu}, one can deduce \eqref{lAEM} by using one of results in \cite{BeV} directly and formally. 
However, the method mentioned in Subsection \ref{Heu} is formal, and we use a different method to prove Theorem \ref{lAEMT}. 
Moreover, any transparent formula for the homogeneous parts of $D(P,f)$ is, in general, not known. 
We will see that, when $P$ is a Delzant lattice polytope, the operators $D_{n}(P;f)$ can be, to some extent, 
expressible concretely (Definition \ref{opcone}, Theorem \ref{BVvsT}). Note that our formula \eqref{lAEM} is valid 
for any smooth function $\varphi$ on $P$. Our construction of the operator $D_{n}(P;f)$ makes us to obtain 
concrete formula for Delzant lattice polytopes in two dimension (Corollary \ref{2Dasympt}). 
A part of our construction of these operators $D_{n}(P;f)$ uses an induction procedure, 
and they are still complicated. This complication comes from the ``angles'' at each face of the polytopes, 
and hence it would be rather natural. The complication involving the ``angles'' is embodied 
in an integration by parts procedure.

In this paper, by the name asymptotic Euler-Maclaurin formula, we mean formulas of asymptotic expansion of 
the Riemann sum $R_{N}(P;\varphi)$. 
In one dimension ($m=1$ and $P=[0,1]$), the following asymptotic Euler-Maclaurin formula is well know. 
\begin{equation}
\label{cAEM}
\begin{split}
\frac{1}{N}\sum_{k=1}^{N}\varphi(k/N)
&=R_{N}([0,1];\varphi)-\frac{\varphi(0)}{N} \\
 \sim \int_{0}^{1}\varphi(x)\,dx+&
\frac{1}{2N}(\varphi(1)-\varphi(0))+
\sum_{n \geq 1}\frac{(-1)^{n-1}B_{n}}{(2n)!}
\left(
\varphi^{(2n-1)}(1)-\varphi^{(2n-1)}(0)
\right)N^{-2n}, 
\end{split}
\end{equation}
where $\varphi$ is any smooth function on $[0,1]$, and $b_{n}$ are the coefficients of the 
Taylor expansion of the Todd function: 
\[
{\rm Todd}(-z)=\sum_{n=0}^{\infty}\frac{b_{n}}{n!}z^{n}, 
\]
and $B_{n}=(-1)^{n-1}b_{2n}$ ($n \geq 1$) are the Bernoulli numbers.

A higher dimensional analogue of \eqref{cAEM} was given by Guillemin-Sternberg (\cite{GS}). 
Namely, Guillemin-Sternberg obtained the asymptotic Euler-Maclaurin formula of the form (assuming that $P$ is Delzant)
\begin{equation}
\label{AEM}
R_{N}(P;\varphi) \sim {\rm Todd}(P;\partial/\partial Nh)
\left.\int_{P_{h}}\varphi(x)\,dx\right|_{h=0}. 
\end{equation}
This formula also holds true for simple lattice polytopes under a modification. 
Note that this formula is, at least its appearance, similar to the Brion-Vergne formula \eqref{eEM}. 
The proof of \eqref{AEM} given in \cite{GS} is different from the proof of \eqref{eEM} given in \cite{BrV}, 
and it does not use geometry of toric varieties. There are some applications of the above formula 
for spectral analysis on toric K\"{a}hler manifolds. In fact, in \cite{GW}, 
the asymptotic Euler-Maclaurin formula obtained in \cite{GS}, combined 
with an asymptotic expansion of `twisted Mellin transform' studied in \cite{W}, is applied 
to analyze a spectral measure on $\mb{C}^{m}$ which is, in a GIT setting, related to 
the pair $(X,L)$ where $X$ is a toric manifold corresponding to a 
Delzant polytope and $L$ is a Hermitian line bundle on $X$. (See also \cite{Ch} where 
the same spectral measure as in \cite{GW} is discussed.)

One more asymptotic Euler-Maclaurin formula was brought to us by Zelditch \cite{Z}. 
The formula obtained in \cite{Z} is stated as 
\begin{equation}
\label{BB}
R_{N}(P;\varphi) \sim \int_{P}\varphi\,dx +\frac{1}{2N}\int_{\partial P}\varphi(x)\,d\sigma +\sum_{n \geq 2}
N^{-n}\int_{P}\ecal_{n}(P)\varphi (x)\,dx, 
\end{equation}
where $P$ is a Delzant lattice polytope, $\ecal_{n}(P)$ is a differential operator (of finite order), and $d\sigma$ is the Leray measure 
on the boundary $\partial P$. 
In \cite{Z}, Zelditch introduced the notion of Bergman-Bernstein measures (this name is taken from \cite{T}) 
and obtained its asymptotic expansion. Then, integration (over the toric K\"{a}hler manifold corresponding to the 
Delzant polytope $P$) of the asymptotic expansion yields the formula \eqref{BB}. 
In \cite{Z}, the formula \eqref{BB} is called a `metric expansion' to distinguish it 
from the Euler-Maclaurin formula of the form \eqref{AEM},  
since the differential operators $\ecal_{n}(P)$ depend on the choice of a Hermitian 
metric on a line bundle over the toric manifold. But, the Riemann sum itself does not depend on 
such a metric. A point is that such a metric dependence would be disappeared 
after an integration by parts. Indeed, in \cite{Z}, the second term is computed by using an 
integration by parts identity due to Donaldson \cite{D}. 

As is mentioned in \cite{Z}, 
comparison of asymptotic Euler-Maclaurin formula and the metric expansion of the form \eqref{BB}
will give some further identities in the lower order terms. 
One of our motivation is to give another asymptotic Euler-Maclaurin formula which is computable to some extent. 
Indeed, we have a concrete formula for the third term of the expansion when the polytope is Delzant. 
See Corollary \ref{thirdT} in Subsection \ref{Comp3}. 
Thus, if one can compute the differential operator $\ecal_{2}(P)$ in \eqref{BB} in terms of curvatures, 
then one will obtain an integration by parts identity in the third term in \eqref{BB}, which might be useful 
to geometry of toric manifolds.

An idea of proof of Theorem \ref{lAEMT} is to reduce the problem to that for unimodular cones, 
which are cones generated by a part of an integral basis, by using a subdivision of a rational cones into 
unimodular cones (see \cite{F}, Section 2.6) and a canonical decomposition of 
the characteristic functions of polytopes (see the equations \eqref{Eu}, \eqref{decoC} in Section \ref{RPC}). 
The asymptotic Euler-Maclaurin formula of Riemann sums over unimodular cones can be deduced by a method  in \cite{GS} 
(see also \cite{AM}, \cite{KSW}, \cite{KSW2}). However, we deduce it here by a quite different method. 
This method is rather similar to the Bergman-Bernstein approach in \cite{Z}. 
But, we work on unimodular cones instead of polytopes themselves. 
Thus, we use the Szasz measures introduced in Section \ref{SZASZF} 
instead of Bernstein or Bergman-Bernstein measures discussed in \cite{Z} or \cite{T}. 
More concretely, an asymptotic property of the Szasz functions is used to show Proposition \ref{uC} in Section \ref{ASCONE}, 
which is an asymptotic Euler-Maclaurin formula for unimodular cones. 
Proposition \ref{uC} can be deduced directly from Theorem 3.2 in \cite{GS}, and  
one can consider that the Proposition \ref{uC} is a starting point for the subsequent sections. 
Thus, one might be able to perform similar computations in sections after Section \ref{ASCONE} at least for 
simple polytopes, by using Theorem 3.3 in \cite{GS} instead of Proposition \ref{uC}. 
However, the asymptotic behavior of Szasz functions would be a general interest in its own right.  
Furthermore, there would be a possibility of using a version of Szasz functions to get 
asymptotics of the Riemann sum over general rational cones without using a subdivision of cones into unimodular cones, 
if one could resolve a problem on `rare events' along with the lines in \cite{T}. 
(See also {\it Remark} after the proof of Theorem \ref{AEMTh} on this point.) 
In one dimension, we compute explicitly each term of the expansion for twisted Riemann sum 
by using this approach. This computation uses the twisted version of the Szasz function, and it shows that 
coefficients in the Taylor expansion of the `twisted' Todd function can be represented by the Stirling numbers of the second kind 
(in particular, the equation \eqref{Bom}), 
which is a generalization of a well-known formula among Bernoulli numbers, Catalan numbers and the Stirling numbers of the 
second kind (see \eqref{Catalan} or \cite{GKP}). Thus, this approach might have 
some advantages also in higher dimension. 
These are the reasons why we use the approach with the Szasz functions in this paper.

We here mention that an asymptotic expansion of the Szasz function was first obtained in \cite{Fe}. 
In \cite{Fe}, Feng also obtained an asymptotic formula of the Riemann sum over the positive orthant $\mb{R}^{m}_{+}$ 
in the same strategy as ours. 
However, concrete formulas for each term of the 
asymptotic expansion are not discussed fully in \cite{Fe}. We give an explicit formula for 
each term of the expansion of the Szasz function in Section \ref{SZASZF}. 
(The main purpose in \cite{Fe} was to give a non-compact analogue of Bergman-Bernstein approximation in \cite{Z}. 
Indeed the Szasz function, defined in Section \ref{SZASZF} in the present paper, is closely related to 
the Bergman kernel for the Bargmann-Fock space as explained in \cite{Fe}.)

We close Introduction with some comments on the organization of this paper. 
We collect some of the notation used in this paper in Subsection \ref{NOT1}, 
and then, we review and define the Berline-Vergne operators $D_{n}(P;f)$ in Subsection \ref{BVOP}. 
As we mentioned above, a heuristic argument to find a formula \eqref{lAEM} is given in Subsection \ref{Heu}. 
In Subsection \ref{RBVDO}, we prove a uniqueness theorem on the expression of each term of the asymptotic 
expansion of the form \eqref{lAEM} (Theorem \ref{BerVer}). 
In Section \ref{SZASZF}, we study asymptotic behavior of Szasz functions. 
Some computations for the twisted Riemann sum in one dimension is given in Subsection \ref{ONED}. 
In Subsections \ref{DEFSF}, \ref{AsymSF}, we define and study Szasz functions and their asymptotic behavior. 
Section \ref{ASCONE} is devoted to the study of asymptotic behavior of the Riemann sums over unimodular cones.  
First, we prove an asymptotic expansion formula (Proposition \ref{uC}) by using 
the asymptotic property of the Szasz functions studied in Section \ref{SZASZF}. 
Asymptotic formula obtained in Proposition \ref{uC} 
uses differential operators in direction transversal to each face of the unimodular cone. 
Then, one can perform further integration by parts. This is done in Subsection \ref{IBP}. 
In Subsection \ref{DOC}, we define differential operators obtained by 
the integration by parts procedure discussed in Subsection \ref{IBP} which is used to 
renormalize each term of the expansion in Proposition \ref{uC}. 
The fact that the operators so defined coincide with the Berline-Vergne operators 
is proved also in this subsection (Theorem \ref{BVvsT}). 
In Section \ref{DRCONE}, we prove the asymptotic Euler-Maclaurin formula for general pointed rational 
cones by using the Berline-Vergne operators and the subdivision of pointed rational cones 
into a finite number of unimodular cones. 
Finally, in Section \ref{RPC}, we prove our main Theorem \ref{lAEMT}, which is reformulated in Theorem \ref{AEMTh}, 
and a uniqueness result (Theorem \ref{UofP}), and give some explicit computation.

\vspace{10pt}

\noindent{\bf Acknowledgment.}\hspace{5pt} The author would like to thank to 
Dr. Micheal Stolz who informed him about the work of O.~Szasz \cite{S}. 
He would also like to thank to Prof. Steve Zelditch for his helpful comments on the 
earlier version of the paper.

\section{Berline-Vergne operators and heuristic argument}
\label{BVOPH}

In this section, we review the symbol of differential operators defined in \cite{BeV}. 
Then, we give a heuristic argument to obtain an asymptotic Euler-Maclaurin formula of the form \eqref{lAEM}. 
Furthermore, we deduce a uniqueness theorem on expression of coefficients in asymptotic Euler-Maclaurin formula  
of the form \eqref{lAEM}.

\subsection{Notation}
\label{NOT1}

Let $X$ be a finite dimensional vector space over $\mb{R}$, and let $\Lambda$ be a lattice in $X$. 
Such a pair $(X,\Lambda)$ is called a rational vector space. 
The dual space $X^{*}$ of a rational space $(X,\Lambda)$ is a rational space 
with the dual lattice $\Lambda^{*}$ of $\Lambda$. 
A point $x \in X$ is said to be rational if $qx \in \Lambda$ for some $q \in \mb{Z} \setminus \{0\}$. 
The set of rational points in $X$ is denoted by $X_{\mb{Q}}$. 
A basis of $\Lambda$ over $\mb{Z}$ is called an integral basis of $\Lambda$. 
For each rational vector space $(X,\Lambda)$, we fix a Lebesgue measure on $X$ normalized so that 
the measure of the fundamental domain of the action of $\Lambda$ on $X$ has measure $1$. 
A subspace $L$ in $X$ is said to be rational if $L \cap \Lambda$ is a lattice in $L$. 
We fix a Lebesgue measure on a rational subspace $(L,L \cap \Lambda)$ as above. 
An affine subspace $A$ is said to be rational if $A$ is a parallel translation of 
a rational subspace. (Note that a rational affine subspace $A$ is allowed to be a translation of 
a rational subspace by a point which is not rational.) 
For a rational affine subspace $A$, we fix a Lebesgue measure 
on $A$ which is a translation of the fixed Lebesgue measure on the rational subspace 
parallel to $A$. 
Any integration on a subset in a rational affine subspace is performed by using the Lebesgue measure 
normalized in this way. For each vector $u \in X$, let $\nabla_{u}$ denote the derivative 
in the direction $u$. 

For each non-empty subset $S$ in $X$, let $L(S)$ be the subspace spanned by the vectors $y-x$ with $x,y \in S$, 
which is parallel to the affine hull, denoted by $\ispa{S}$, of $S$. 
If $S \subset X_{\mb{Q}}$, then $L(S)$ is a rational subspace in $X$. 
Let $L$ be a rational subspace in a rational space $(X,\Lambda)$. 
The natural projection from $X$ onto $X/L$ is denoted by $\pi_{L}:X \to X/L$. 
If $L$ is a subspace in $X$, let $L^{\perp} \subset X^{*}$ denote the annihilator of $L$. 
The quotient space $X/L$ of $X$ by a rational subspace $L$ is again a rational space with the lattice $\pi_{L}(\Lambda)$. 

An inner product $Q$ on a rational space $(X,\Lambda)$ is said to be rational 
if $Q(x,y) \in \mb{Q}$ for each $x,y \in X_{\mb{Q}}$. 
Let $Q$ be a rational inner product on $(X,\Lambda)$. 
The rational inner product on $(X^{*},\Lambda^{*})$ 
induced by the inner product $Q$ on $X$ is also denoted by $Q$. 
Let $L$ be a subspace in $X$. 
The orthogonal complement of $L$ in $X$ is denoted by $L^{\perp_{Q}}$. 
Note that we have a natural identification $(X/L)^{*} \cong L^{\perp}$. 
The orthogonal projection from $X^{*}$ onto $(X/L)^{*} \cong L^{\perp}$ is denoted by 
$p_{L}:X^{*} \to (X/L)^{*}$. When $L$ is rational, the rational space $X/L$ is equipped with 
the rational inner product obtained by identifying $X/L$ with $L^{\perp_{Q}}$. 
Note that, with this identification, the lattice $\pi_{L}(\Lambda)$ of $X/L$ is identified with the 
orthogonal projection $p_{L}(\Lambda)$ of $\Lambda$, where the orthogonal projection 
from $X$ onto $L^{\perp_{Q}}$ is also denoted by $p_{L}:X \to L^{\perp_{Q}}$, which 
is different from the lattice $L^{\perp_{Q}} \cap \Lambda$ in $L^{\perp_{Q}}$. 

A subset $P$ in a rational space $(X,\Lambda)$ is called a rational polyhedron 
if $P$ is an intersection of a finite number of half spaces each of which is bounded by a rational affine hyperplane. 
Let $P$ be a rational polyhedron. Then the set of faces of $P$ is denoted by $\fcal(P)$, 
and, for non-negative integer $k$, the set of faces of $P$ of dimension $k$ is denoted by $\fcal(P)_{k}$. 
We set $\vcal(P)=\fcal(P)_{0}$, the set of vertices of $P$. 
A face of codimension one is called a facet. 
For each $f \in \fcal(P)$, we set $\pi_{f}=\pi_{L(f)}$, the natural projection from $X$ onto $X/L(f)$. 
When, a rational inner product on $X$ is fixed, we set $p_{f}=p_{L(f)}$, the orthogonal 
projection from $X^{*}$ onto $(X/L(f))^{*}$. 
A rational polyhedron $C$ in $X$ is called a rational cone if $C$ is a cone generated 
by a finite number of elements in $\Lambda$. Note that a rational cone $C$ might contain straight lines. 
The largest subspace contained in the rational cone $C$ is $C \cap (-C)$, which is a rational subspace in $X$. 
If $C \cap (-C)=\{0\}$, then the rational cone $C$ is said to be pointed. 
If a rational cone $C$ is generated by a subset of 
an integral basis of $\Lambda$, then $C$ is said to be unimodular. 
A subset $\mf{a}$ of $X$ is called a rational affine cone if $\mf{a}$ is of the form $\mf{a}=s +C$ where 
$s \in X_{\mb{Q}}$ and $C$ is a rational cone. If $C$ is pointed, then $\mf{a}$ is also said to be pointed.

\subsection{The Berline-Vergne operators}
\label{BVOP}

In this subsection, we recall the construction of operators given in \cite{BeV}. 
Let $(X,\Lambda)$ be a rational space with a rational inner product $Q$. 
For each rational polyhedron $P$ in $X$, we set 
\begin{equation}
\label{SandI}
S(P)(\xi)=\sum_{\gamma \in P \cap \Lambda}e^{\ispa{\xi,\gamma}},\quad 
I(P)=\int_{P}e^{\ispa{\xi,x}} 
\end{equation}
if the sum and the integral converge absolutely, where $\xi \in X^{*}$. 
These functions are defined as meromorphic functions on $X^{*}$. 
Let $f$ be a face of a rational polyhedron $P$ in $X$. 
Let $C_{P}(f)$ be the cone generated by the vectors of the form $y-x$ with $y \in P$, $x \in f$. 
This is actually a rational cone in $X$ with $C_{P}(f) \cap (-C_{P}(f))=L(f)$. 
Then, the pointed affine cone $\mf{t}(P,f):=\pi_{f}(\ispa{f}+C_{P}(f))$ in $X/L(f)$ is called the transverse cone of $P$ along $f$. 

For any rational quotient $W=X/L$ of $X$ by a rational subspace $L$, 
let $\ccal(W)$ denote the set of all rational affine cones in $W$. 
Let $\hcal(W^{*})$ denote the ring of analytic functions with rational Taylor coefficients defined in a neighborhood of $0$ in $W^{*}$ 
with respect to an (and hence all) integral basis of the dual lattice of the lattice $\pi_{L}(\Lambda)$ in $W=X/L$. 

Then, it is shown in Theorem 20 in \cite{BeV} that there is a unique family of maps $\mu_{W}$, indexed by rational quotient spaces $W$ of $X$, 
from $\ccal(W)$ to $\hcal(W^{*})$ such that the following conditions hold: 
\begin{enumerate}
\item If $W=\{0\}$, then $\mu_{W}(\{0\})=1$. 
\item If the affine cone $\mf{a} \in \ccal(W)$ contains a straight line, then $\mu_{W}(\mf{a})=0$. 
\item For any $\mf{a} \in \ccal(W)$, one has 
\begin{equation}
\label{ind1}
S(\mf{a})(\xi)=\sum_{F \in \fcal(\mf{a})}
\mu_{W/L(F)}(\mf{t}(\mf{a},F))(\xi)I(F)(\xi),\quad \xi \in W^{*}.   
\end{equation}
\end{enumerate}
Moreover, one of main theorems in \cite{BeV} is that, for each rational polyhedron $P$ in $W=X/L$, 
one has 
\begin{equation}
\label{ind2}
S(P)(\xi)=\sum_{f \in \fcal(P)}\mu_{X/L(f)}(\mf{t}(P,f))(\xi)I(f)(\xi), \quad \xi \in W^{*}. 
\end{equation}
(See Theorem 21 in \cite{BeV}.) 
Note that the functions $\mu_{X/L(f)}$ in \eqref{ind2} (and also in \eqref{ind1}) 
is the lift to $W^{*}$ of functions defined on $(W/L(f))^{*}$ through the orthogonal projection $p_{f}:W^{*} \to (W/L(f))^{*}$. 
Let $\mf{a}$ be a pointed rational affine cone in the rational quotient $X/L$ of $X$. 
For any non-negative integer $k$, let $\mu_{X/L}^{k}(\mf{a})$ denote the homogeneous 
polynomial of degree $k$ on $(X/L)^{*}$ which is the homogeneous part of 
the Taylor expansion of the analytic function $\mu_{X/L}(\mf{a})$ near $0 \in (X/L)^{*}$. 
We set $\mu_{X}^{k}(\mf{a})=p_{L}^{*}\mu_{X/L}^{k}(\mf{a})$, which is a homogeneous 
polynomial of degree $k$ on $X^{*}$.

\begin{defin}
\label{opconeG}
Let $(X,\Lambda)$ be a rational space with a rational inner product $Q$. 
For any rational polyhedron $P$ in $X$, any face $f$ of $P$ and any non-negative integer $n$ such that $n-\dim(P)+\dim(f) \geq 0$, 
we define the homogeneous differential operator $D_{n}^{X}(P;f)$ on $X$ with rational constant coefficients 
of order $n-\dim(P)+\dim(f)$, which involves derivatives only in directions perpendicular to the subspace $L(f)$, 
as the differential operator whose symbol is given by 
$\mu_{X}^{n-\dim(P)+\dim(f)}(\mf{t}(P,f))=p_{f}^{*}\mu_{X/L(f)}^{n-\dim(P)+\dim(f)}(\mf{t}(P,f))$.
We call the operators $D_{n}^{X}(P;f)$ the Berline-Vergne operators. 
\end{defin}
We note that, when $C$ is a pointed rational cone in $X$ and $F$ is a face of $C$, 
then $\mf{t}(C,F)=\pi_{F}(C)$, and hence we have $D_{n}^{X}(C;F)=D_{n}^{X}(\pi_{F}(C);0)$. 
Let $P$ be a lattice polytope in $X$, which means that each vertex is an element in $\Lambda$, 
and let $f \in \fcal(P)$. Then, we have $\mf{t}(P,f)=\pi_{f}(v)+\pi_{f}(C_{P}(f))$ where $v \in f \cap \Lambda$. 
Since the function $\mu_{X/L(f)}$ is invariant under translation by elements in the lattice (Theorem 21 in \cite{BeV}), 
we have $D_{n}^{X}(P;f)=D_{n}^{X}(\pi_{f}(C_{P}(f));0)$.

\subsection{Heuristic arguments}
\label{Heu}

In this subsection, we give a heuristic argument to find the formula \eqref{lAEM} 
by using the result \eqref{ind2} in \cite{BeV}. 
Let $(X,\Lambda)$ be a rational space. Let $P$ be a lattice polytope in $X$. 
For simplicity, assume that $m:=\dim(P)=\dim(X)$. 
For each $f \in \fcal(P)$, we set $\mu(P,f):=p_{f}^{*}\mu_{X/L(f)}(\mf{t}(P,f))$ which 
is a meromorphic function on $X^{*}$ analytic in a neighborhood of the origin.  
Now let us compute the Riemann sum $R_{N}(P;\varphi)$ by using \eqref{ind2}. 
Let $\varphi$ be a smooth function on $P$. Since $P$ is compact, one may assume that $\varphi \in C_{0}^{\infty}(X)$. 
Normalize the Lebesgue measure $d\xi$ on $X^{*}$ so that it satisfy the Fourier inversion formula
\[
\varphi(x)=(2\pi)^{-m}\int_{X^{*}}e^{i\ispa{\xi,x}}\hat{\varphi}(\xi)\,d\xi,\quad 
\hat{\varphi}(\xi)=\int_{X}e^{-i\ispa{\xi,x}}\varphi(x). 
\] 
Inserting the above for $x=\gamma/N$ with $\gamma \in NP \cap \Lambda$ 
into the definition of $R_{N}(P;\varphi)$ and using the formula \eqref{ind2}, we have 
\[
R_{N}(P;\varphi)=\frac{1}{(2\pi N)^{m}}\sum_{f}\int_{X^{*}}
\mu(NP,Nf)(i\xi/N)I(Nf)(i\xi/N)\widehat{\varphi}(\xi)\,d\xi,  
\]
But, since $P$ is a lattice polytope, we have $\mu(NP,Nf)=\mu(P,f)$ (see \cite{BeV}, Remark 29). 
Changing the variable $x \mapsto x/N$, we have $I(Nf)(i\xi/N)=N^{\dim(f)}I(f)(i\xi)$. 
Thus we have 
\[
R_{N}(P;\varphi)=
\frac{1}{(2\pi N)^{m}}
\sum_{f}N^{\dim(f)}\int_{X^{*}}
\mu(P,f)(i\xi/N)I(f)(i\xi)\widehat{\varphi}(\xi)\,d\xi. 
\]
Formally, substituting the Taylor expansion 
\[
\mu(P,f)(i\xi/N)=\sum_{k \geq 0}\mu^{k}(\mf{t}(P,f))(i\xi)N^{-k}
\]
into the above formula, we could have 
\begin{equation}
\label{formal}
R_{N}(P;\varphi) \mbox{``$\sim$''}\sum_{n \geq 0}N^{-n}
\sum_{f \in \fcal(P)\,;\,\dim(f) \geq m-n}
\int_{f}D_{n}^{X}(P;f)\varphi, 
\end{equation}
where $D_{n}^{X}(P;f)$ is defined in Definition \ref{opconeG}. 
However, the above computation is formal because we do not know much about global properties of the functions $\mu(P,f)$. 
Even if we could prove the formula \eqref{formal} along with the method explained above, we do not know much about homogeneous parts of 
its Taylor expansion. 
One of our purposes in this paper is to give an effective formula for the operator $D_{n}^{X}(P;f)$ given in Definition \ref{opconeG}, 
at least for Delzant lattice polytopes, by a method different from the above strategy.

\subsection{A uniqueness property}
\label{RBVDO}

In this subsection, we discuss a uniqueness property of an expression of each term of the asymptotic 
expansion of $R_{N}(C;\varphi)$ for unimodular cones $C$. 
Let $C$ be a unimodular cone in a rational space $(X,\Lambda)$ with a rational inner product $Q$. 
Then, note that, for each face $F$ of $C$, we have $\mf{t}(C,F)=\pi_{F}(C)$.
Note also that, we give a rational inner product in each rational quotient space $X/L$ 
by identifying $X/L$ with $L^{\perp_{Q}}$.

\begin{theorem}
\label{BerVer}
Suppose that, for any rational space $(X,\Lambda)$ with a rational inner product $Q$, 
any rational subspace $L$ of $X$, any unimodular cone $C$ in $X/L$ and any non-negative integer $n$ such that 
$n \geq \dim(C)$, there exists a homogeneous differential operator $\dcal_{n}^{X}(C)$ on $X$ of order $n-\dim(C)$ with symbol $\nu_{n}^{X}(C)$ such that 
\begin{enumerate}
\item If $C \subset X/L$, then $\nu_{n}^{X}(C)=p_{L}^{*}\nu_{n}^{X/L}(C)$ where $p_{L}:X \to L^{\perp_{Q}} \cong (X/L)^{*}$ denote 
the orthogonal projection. 
\item If $C \subset X$ with $\dim(C) < \dim(X)$, then $\nu_{n}^{X}(C)=\!\,^{t}\iota_{C}^{*}\nu_{n}^{L(C)}(C)$, where $\,\!^{t}\iota_{C}:X^{*} \to L(C)^{*}$ 
is the transpose of the inclusion $\iota_{C}:L(C) \hookrightarrow X$. 
\item When $\dim(X)=0$, we have $\dcal_{0}^{X}(\{0\})=1$, $\dcal_{n}^{X}(\{0\})=0$ $(n \geq 1)$. 
When $\dim(X)=1$ and $C=\mb{R}_{+}u$ with a generator $u$ of $\Lambda$, we have $\dcal_{n}^{X}(C)=-\frac{b_{n}}{n!}\nabla_{u}^{n-1}$ $(n \geq 1)$. 
\item For any unimodular cone $C \subset X$, any $F \in \fcal(C)$, any $n \in \mb{Z}_{+}$ with $\dim(F) \geq \dim(C)-n$ 
and any Schwartz function $\varphi \in \scal(X)$ on $X$, the following holds$:$
\begin{equation}
\label{Ascone}
R_{N}(C;\varphi) \sim \sum_{n \geq 0}N^{-n}\sum_{F \in \fcal(C)\,;\,\dim(F) \geq \dim(C)-n}
\int_{F}\dcal_{n}^{X}(\pi_{F}(C))\varphi \quad (N \to \infty). 
\end{equation}
\end{enumerate}
Then, we have 
\begin{equation}
\label{BVft}
\nu_{n}^{X}(C)=\mu_{X}^{n-\dim(C)}(C)
\end{equation}
for any such $X$, $C$ and $n$ satisfying $n-\dim(C) \geq 0$. 
\end{theorem}

\begin{proof}
First, we note that, the symbols of the Berline-Vergne operators satisfy the assumption (2) in the statement (Proposition 13 in \cite{BeV}). 

We prove the assertion by induction on the dimension of $X$. 
Consider the case where $\dim X=1$. Take a generator $u$ of the lattice $\Lambda$ and 
identify $u$ with $1$ in $\mb{Z}$. Let $C=\mb{R}_{+}u$. 
Then, as is computed in \cite{BeV}, we have 
\[
\mu_{X}(C)(\xi)=\frac{1}{\ispa{\xi,u}}+\frac{1}{1-e^{\ispa{\xi,u}}}=-\sum_{n=1}^{\infty}\frac{b_{n}}{n!}\ispa{\xi,u}^{n-1},\quad \xi \in X^{*}. 
\]
We also have $\mu_{\{0\}}(\{0\})=1$. 
From this, we have $\mu_{X}^{n-1}(C)(\xi)=-\frac{b_{n}}{n!}\ispa{\xi,u}^{n-1}$ ($n \geq 1$), $\mu_{X}^{0}(\{0\})=1$, $\mu_{X}^{n}(\{0\})=0$ ($n \geq 1$). 
By the assumption (3), this shows the assertion when $\dim(X)=1$. 

Next, assume that, for any rational space $(X,\Lambda)$ with $\dim(X) \leq m-1$, 
any unimodular cone $C$ in a rational quotient $X/L$ and any non-negative integer $n$ such that $n \geq \dim(C)$, 
the equation \eqref{BVft} holds. 
Let $X$ be an $m$-dimensional rational space, and let $C \subset X$ be a unimodular cone. 
If $\dim(C) < m$, then by the assumption (2) and the induction hypothesis, we have \eqref{BVft}. 
Thus, we assume that $\dim(C)=m$. 
Let $F \in \fcal(C)$. If $\dim(F)>0$, then, by the assumption (1), we have 
$\nu_{n}^{X}(\pi_{F}(C))=p_{F}^{*}\nu_{n}^{X/L(F)}(\pi_{F}(C))$. 
Since $\dim(X/L(F)) \leq m-1$ and $\pi_{F}(C)$ is a unimodular cone in $X/L(F)$, we can use 
the induction hypothesis, and hence the latter function coincides with $p_{F}^{*}\mu_{X/L(F)}^{n-m+\dim(F)}(\pi_{F}(C))=\mu_{X}^{n-m+\dim(F)}(\pi_{F}(C))$. 
To prove $\nu_{n}^{X}(C)=\mu_{X}^{n-m}(C)$ for $n \geq m$, 
take $\xi \in X^{*}$ such that $\ispa{\xi,x}<0$ for each $x \in C$. 
Then, for any $N >0$, we have 
\[
S(C)(\xi/N)=N^{m}R_{N}(C;e_{\xi}),\quad e_{\xi}(x)=e^{\ispa{\xi,x}}. 
\]
Note that there is a $\varphi \in \scal(X)$ such that $\varphi(x)=e_{\xi}(x)$ for $x \in C$. 
Thus, by the assumption (4), we have 
\begin{equation}
\label{ind11}
S(C)(\xi/N)\sim \sum_{n \geq 0}N^{m-n}\sum_{F \in \fcal(C)\,;\,\dim(F) \geq m -n}\nu_{n}^{X}(\pi_{F}(C))(\xi)I(F)(\xi) 
\end{equation}
as $N \to \infty$. 
By \eqref{ind1} and the identity $I(F)(\xi/N)=N^{\dim(F)}I(F)(\xi)$, we have
\[
S(C)(\xi/N)=\sum_{n \geq 0}N^{m-n}\sum_{F \in \fcal(C)\,;\,\dim(F) \geq m-n}\mu_{X}^{n-m+\dim(F)}(\pi_{F}(C))(\xi)I(F)(\xi) 
\]
for every sufficiently large $N$. 
Let $n \geq m$. By using the induction hypothesis, 
the coefficient of $N^{m-n}$ in the above can be written as 
\begin{equation}
\label{ind22}
\mu_{X}^{n-m}(C)(\xi)+\sum_{F \in \fcal(C)\,;\,0 \neq \dim(F) \geq m-n}\nu_{n}^{X}(\pi_{F}(C))(\xi)I(F)(\xi). 
\end{equation}
Equating \eqref{ind22} with the coefficient of $N^{m-n}$ in \eqref{ind11} shows $\mu_{X}^{n-m}(C)=\nu_{n}^{X}(C)$. 
\end{proof}

\section{Szasz functions and their asymptotic behavior}
\label{SZASZF}

In this section, we define Szasz functions over unimodular cones and investigate their asymptotic behavior. 
First of all, let us compute in one dimension, which illustrate the general case.

\subsection{Computation in one dimension}
\label{ONED}

The Szasz function associated with a function $\varphi$ on $\mb{R}$, originally introduced and discussed in \cite{S}, is defined by 
\begin{equation}
\label{Sf}
S_{N}(\varphi)(x)=\sum_{k=0}^{\infty}\ell_{k}(Nx)\varphi(k/N),\quad \ell_{k}(x)=\frac{x^{k}}{k!}e^{-x},\quad x \in \mb{R}. 
\end{equation}
Szasz introduced the function $S_{N}(\varphi)$ as an analogue of the Bernstein polynomial
\[
B_{N}(\varphi)(x)=\sum_{k=0}^{N}m_{N}^{k}(x)\varphi(k/N),\quad 
m_{N}^{k}(x)={N \choose k}x^{k}(1-x)^{N-k}. 
\]
Indeed, these two functions are related through Poisson's law of rare events
\[
\lim_{N \to \infty}m_{N}^{k}(x/N)=\ell_{k}(x). 
\]
For us, an important property of the Szasz function $S_{N}(\varphi)$ is the following: 
\[
\int_{0}^{\infty}S_{N}(\varphi)(x)\,dx=\frac{1}{N}\sum_{k=0}^{\infty}\varphi(k/N)=:R_{N}([0,+\infty);\varphi)
\]
for any $\varphi \in \scal(\mb{R})$. We put 
\[
R_{N}((-\infty,0];\varphi):=\frac{1}{N}\sum_{k=0}^{\infty}\varphi(-k/N). 
\]
Then, once we obtain the asymptotic expansion of $S_{N}(\varphi)$ as $N \to \infty$ with a 
suitable reminder estimate, then integrating it on $[0,\infty)$ will give the asymptotic expansion of $R_{N}([0,+\infty);\varphi)$. 
But then we have the formula 
\begin{equation}
\label{alt1}
R_{N}([0,1];\varphi)=R_{N}([0,+\infty);\varphi) +R_{N}((-\infty,0];T_{1}\varphi)-R_{N}(\mb{R};\varphi), 
\end{equation}
where we set $T_{1}\varphi (x)=\varphi(1+x)$. 
In this formula, note that we have $R_{N}(\mb{R};\varphi)=\int_{\mb{R}}\varphi(x)\,dx+O(N^{-\infty})$ (see \cite{GS} or see Lemma \ref{notuC}). 
We also have $R_{N}((-\infty,0];T_{1}\varphi)=R_{N}([0,+\infty);\psi)$, where we set $\psi(x)=\varphi(1-x)$, 
and hence the asymptotics of $R_{N}([0,+\infty);\varphi)$ will give the classical asymptotic 
Euler-Maclaurin formula \eqref{cAEM}. 
Thus, to obtain \eqref{cAEM}, it is enough to consider $R_{N}([0,+\infty);\varphi)$. 
In one dimension, we can consider a bit more general situation. 
We choose a positive integer $q \geq 1$ and a $q^{{\rm th}}$ root of unity $\omega$. 
We consider the twisted Riemann sum 
\begin{equation}
\label{twist}
R_{N}^{\omega}(\varphi):=\frac{1}{N}\sum_{k=0}^{\infty}\omega^{k}\varphi(k/N),  
\end{equation}
where $\varphi \in C_{0}^{\infty}(\mb{R})$. The twisted Riemann sum $R_{N}^{\omega}(\varphi)$ is 
discussed in \cite{GS} and the asymptotic formula 
\begin{equation}
\label{twistEM}
R_{N}^{\omega}(\varphi) \sim 
\sum_{n \geq 1}(-1)^{n-1}b_{n}^{\omega}\frac{\varphi^{(n-1)}(0)}{N^{n}}
\end{equation}
was obtained, where the coefficients $b_{n}^{\omega}$ is defined by the Taylor expansion of the function 
\begin{equation}
\label{twistTodd}
\tau_{\omega}(s):=\frac{s}{1-\omega e^{-s}}=\sum_{n \geq 1}b_{n}^{\omega}s^{n},\quad b_{1}^{\omega}=\frac{1}{1-\omega}. 
\end{equation}
The formula \eqref{twistEM} is used in \cite{GS} to obtain asymptotic Euler-Maclaurin formula for simple polytopes. 
Now, to obtain the asymptotic expansion of the twisted Riemann sum $R_{N}^{\omega}(\varphi)$ along with 
our strategy, we use the twisted version of the Szasz function, which is defined by 
\begin{equation}
\label{twistSz}
S_{N}^{\omega}(\varphi)(x)=\sum_{k=0}^{\infty}\omega^{k}\ell_{k}(Nx)\varphi(k/N). 
\end{equation}
From the definition, we have 
\begin{equation}
\label{twSR}
\int_{0}^{\infty}S_{N}^{\omega}(\varphi)(x)\,dx=R_{N}^{\omega}(\varphi). 
\end{equation}
To state a result on asymptotic expansion of the twisted Szasz function $S_{N}^{\omega}(\varphi)$, we need to 
prepare some properties of the Stirling numbers of the second kind and related polynomials. 

The Stirling numbers of the second kind, denoted by $S(n,k)$ where $n,k$ are integers satisfying $0 \leq k \leq n$, 
are defined by the following recursion formula: 
\begin{equation}
\label{stirling}
\begin{gathered}
S(0,0)=1,\quad S(n,0)=0,\quad S(n,n)=1\ \ (n \geq 1),\\
S(n+1,k)=kS(n,k)+S(n,k-1)\ \ (1 \leq k \leq n). 
\end{gathered}
\end{equation}
For example, we have $S(n,1)=1$ $(n \geq 1)$ and $S(n,n-1)={n \choose 2}$ $(n \geq 2)$. 
For convenience, we set $S(n,k)=0$ for $0 \leq n <k$. For any integer $n,k$ with $0 \leq k \leq n$, 
we define the polynomial $p(n,k;z)$ in $z \in \mb{C}$ of degree $k$ by 
\begin{equation}
\label{stpoly}
p(n,k;z):=\sum_{t=0}^{k}{n \choose t}(-1)^{t}S(n-t,k-t)z^{k-t}. 
\end{equation}
Some of $p(n,k;z)$ are computed as follows. 
\begin{equation}
\label{special1}
\begin{gathered}
p(0,0;z)=1,\quad p(n,0;z)=0,\quad p(n,n;z)=(z-1)^{n}\ \ (n \geq 1) \\
p(n,1;z)=z,\quad p(n,n-1;z)={n \choose 2}z(z-1)^{n-2}\ \ (n \geq 2). 
\end{gathered}
\end{equation}

\begin{lem}
\label{spoly}
\begin{enumerate}
\item For any non-negative integer $n$, we have 
\[
e^{z}\sum_{k=0}^{n}S(n,k)z^{k}=\sum_{k=0}^{\infty}\frac{k^{n}}{k!}z^{k}. 
\]
\item The polynomials $p(n,k;z)$ satisfy the following recursion formula$:$ 
\[
p(n+1,k;z)=(z-1)p(n,k-1;z)+kp(n,k;z)+np(n-1,k-1;z),\quad 1 \leq k \leq n. 
\]
\item For $[n/2]+1 \leq k \leq n$, the polynomial $p(n,k;z)$ is divisible by $(z-1)^{2k-n}$. 
In particular, we have $p(n,k;1)=0$ for $[n/2]+1 \leq k \leq n$. 
\end{enumerate}
\end{lem}

\begin{proof}
(1) is proved easily by using induction on $n$ and the recurrence formula for the 
Stirling numbers $S(n,k)$ of the second kind. 
To prove (2), let $1 \leq k \leq n$. By using the relation ${n+1 \choose t}={n \choose t}+{n \choose t-1}$ for $1 \leq t \leq n$,  
we have 
\[
p(n+1,k;z)=\sum_{t=0}^{k}{n \choose t}(-1)^{t}S(n+1-t,k-t)z^{k-t}-p(n,k-1;z). 
\]
Denote $S$ the sum above. Then, by the recursion formula \eqref{stirling}, we have 
\[
\begin{split}
S&=\sum_{t=0}^{k}{n \choose t}(-1)^{t}(k-t)S(n-t,k-t)z^{k-t}+\sum_{t=0}^{k-1}{n \choose t}(-1)^{t}S(n-t,k-1-t)z^{k-t} \\
&=kp(n,k;z)-n\sum_{t=1}^{k}{n-1 \choose t-1}(-1)^{t}S(n-t,k-t)z^{k-t}+zp(n,k-1;z). 
\end{split}
\]
Minus the sum in the middle of the above equals $np(n-1,k-1;z)$, and hence (2) is proved. 

Let us prove (3). Since the statement is obvious from \eqref{special1} for $n=1,2$, we assume that, for some $n \geq 2$, 
$p(m,k;z)$ is divisible by $(z-1)^{2k-m}$ for each $1 \leq m \leq n$ and $[m/2]+1 \leq k \leq m$, 
and use the induction on $n$. So, we take $l$ with $[(n+1)/2]+1 \leq l \leq n+1$. 
If $l=n+1$, $p(n+1,n+1;z)=(z-1)^{n+1}$ and hence (3) is clear. Thus, we assume that $[(n+1)/2]+1 \leq l \leq n$. 
By the induction hypothesis, $p(n,l;z)$ is divisible by $(z-1)^{2l-n}$. 
We have $[(n-1)/2]+1 = [(n+1)/2]$ and hence, by induction hypothesis, $p(n-1,l-1;z)$ is divisible by $(z-1)^{2l-n-1}$. 
If $[n/2]=l-1$, then $n$ is even and $2l-n-1=1$, and hence, by the recurrence relation (2), $p(n+1,l;z)$ is 
divisible by $(z-1)$. Otherwise, we have $[n/2]+1 \leq l-1$, and hence $p(n,l-1;z)$ is 
divisible by $(z-1)^{2l-n-2}$. Then, again by (2), $p(n+1,l;z)$ is divisible by $(z-1)^{2l-n-1}$. 
\end{proof}

Now, we can state the asymptotic expansion of the twisted Szasz functions 
$S_{N}^{\omega}(\varphi)$ by using the polynomials $p(n,k;z)$ as follows. 

\begin{prop}
\label{tSzA}
Let $\varphi \in \scal(\mb{R})$. Let $\omega$ be a $q^{{\rm th}}$ root of unity. 
Then, for any positive integer $n$ and positive number $K$ such that $n < K < 2n$, 
there exists a constant $C_{K,n}>0$ such that we have 
\begin{equation}
\label{as11}
S_{N}^{\omega}(\varphi)(x)=\sum_{\mu=0}^{2n-1}\frac{\varphi^{(\mu)}(x)}{\mu!}
N^{-\mu}J_{\mu}^{\omega}(Nx)+S_{2n,N}^{\omega}(x),\quad x > 0, 
\end{equation}
where 
the function $S_{2n,N}^{\omega}(x)$ satisfies the following estimate$:$
\begin{equation}
\label{er}
|S_{2n,N}^{\omega}(x)| \leq C_{K,n}N^{-n}(1+x)^{n-K},\quad x>0, \ N>0.  
\end{equation}
The function $J_{\mu}^{\omega}(x)$ is given by 
\begin{equation}
\label{jo}
J_{\mu}^{\omega}(x)=e^{-(1-\omega)x}\sum_{k=0}^{\mu}p(\mu,k;\omega)x^{k}.  
\end{equation}
When $\omega=1$, the function $J_{\mu}^{1}(x)$ is a polynomial in $x$ of degree at most $[\mu/2]$. 
\end{prop}

\begin{proof}
Let $x>0$. Substituting the Taylor expansion 
\[
\begin{gathered}
\varphi(k/N)=\sum_{0 \leq \mu \leq 2n-1}\frac{\varphi^{(\mu)}(x)}{\mu!}(k/N-x)^{\mu}+
\frac{(k/N-x)^{2n}}{(2n-1)!}R_{2n}(k/N,x), \\
R_{2n}(k/N,x)=\int_{0}^{1}(1-t)^{2n-1}\varphi^{(2n)}(x+t(k/N-x))\,dt, 
\end{gathered}
\]
we have 
\[
S_{N}^{\omega}(\varphi)(x)=\sum_{\mu=0}^{2n-1}\frac{\varphi^{(\mu)}(x)}{\mu!N^{\mu}}J_{\mu}^{\omega}(Nx)+S_{2n,N}^{\omega}(x), 
\]
where $J_{\mu}^{\omega}(x)$ and $S_{2n,N}(x)$ are given by 
\[
\begin{gathered}
J_{\mu}^{\omega}(x)=\sum_{k=0}^{\infty}\omega^{k}\ell_{k}(x)(k-x)^{\mu},\\
S_{2n,N}^{\omega}(x)=\frac{1}{(2n-1)!N^{2n}}\sum_{k=0}^{\infty}\omega^{k}\ell_{k}(Nx)(k-Nx)^{2n}R_{2n}(k/N,x). 
\end{gathered}
\]
By using Lemma \ref{spoly}, (1) and the definition \eqref{Sf} of the function $\ell_{k}(x)$, 
it is easy to show the formula \eqref{jo} for $J_{\mu}^{\omega}(x)$. 
We set $S_{2n,N}^{\omega}(x)=\frac{1}{(2n-1)!N^{2n}}S_{2n,N}(x)$. 
Take $K$ as in the statement, and choose $C>0$ so that $|\varphi^{(2n)}(y)| \leq C(1+|y|)^{-K}$ for any $y \in \mb{R}$. 
Then, we have $|x+t(k/N-x)| \geq (1-t)x$ for any $t \in [0,1]$, $x \geq 0$, $k \geq 0$, and hence 
$|R_{2n}(k/N,x)| \leq C_{K,n}x^{-K}$, $x > 1$, $k \geq 0$. 
When $0\leq x \leq 1$, $|R_{2n}(k/N,x)|$ is bounded uniformly in $N$. 
Thus, we have $|S_{2n,N}(x)| \leq Cx^{-K}J_{2n}^{1}(Nx)$ for $x>1$. 
When $0 \leq x \leq 1$, we have $|S_{2n,N}(x)| \leq CN^{-2n}J_{2n}^{1}(Nx)$. 
But, by Lemma \ref{spoly}, (3) and the formula \eqref{jo}, $J_{2n}^{1}(x)$ is a polynomial 
in $x$ of degree at most $n$. Therefore, we obtain \eqref{er}. 
\end{proof}

In general, for any $\tau \in \mb{C}$ with $\re(\tau)>0$ and any $n >0$, we have 
\[
\int_{0}^{\infty}e^{-\tau Nx}\varphi(x)\,dx =\sum_{j=1}^{n-1}\frac{\varphi^{(j-1)}(0)}{(\tau N)^{j}} +O(N^{-n}). 
\]
Taking $K>0$ in Proposition \ref{tSzA} so that $n+1<K<2n$ and integrating \eqref{as11}, we conclude the following. 

\begin{prop}
When $\omega \neq 1$ is the $q^{{\rm th}}$ root of unity, we have 
\begin{equation}
\label{twEMn1}
R_{N}^{\omega}(\varphi) \sim \sum_{n \geq 1}c_{n}^{\omega}\frac{\varphi^{(n-1)}(0)}{N^{n}},\quad
c_{n}^{\omega}=\sum_{\alpha=0}^{n-1}\sum_{k=0}^{\alpha}
\frac{(n-k-1)!}{\alpha!(n-\alpha-1)!}\frac{p(\alpha,\alpha-k;\omega)}{(1-\omega)^{n-k}}. 
\end{equation}
When $\omega=1$, we have 
\begin{equation}
\label{EMn2}
R_{N}([0,\infty);\varphi) \sim \int_{0}^{\infty}\varphi(x)\,dx +\sum_{n \geq 1}c_{n}\frac{\varphi^{(n-1)}(0)}{N^{n}}, \quad
c_{n}=\sum_{\alpha=n}^{2n}\frac{(\alpha-n)!}{\alpha!}(-1)^{\alpha-n+1}p(\alpha,\alpha-n), 
\end{equation}
where we set 
\begin{equation}
\label{pnk}
p(n,k):=p(n,k;1)=\sum_{t=0}^{k}
{n \choose t}(-1)^{t}S(n-t,k-t),\quad 0 \leq k \leq n.  
\end{equation}
\end{prop}
Note that a direct computation and the well-known formula for the relation among the Bernoulli numbers, 
Catalan numbers $\frac{1}{n+1}{2n \choose n}$, and the Stirling numbers (\cite{GKP}) shows
\begin{equation}
\label{Catalan}
c_{n}=-\frac{(n+1)}{n!}{2n \choose n}^{-1}\sum_{l=0}^{n}\frac{(-1)^{l}}{l+1}{2n \choose n+l}S(n+l,l)=-\frac{b_{n}}{n!}, 
\end{equation}
which shows that, for $\omega=1$, we have 
\begin{equation}
R_{N}([0,\infty);\varphi) \sim \int_{0}^{\infty}\varphi(x)\,dx - \sum_{n \geq 1}\frac{b_{n}}{n!}\varphi^{(n-1)}(0)N^{-n}, 
\end{equation}
from which we have \eqref{cAEM}. For $\omega \neq 1$, we compare each term of the asymptotics \eqref{twistEM}, \eqref{twEMn1} to get 
\begin{equation}
\label{Bom}
b_{n}^{\omega}=(-1)^{n-1}c_{n}^{\omega}
=\sum_{\alpha=0}^{n-1}\sum_{k=0}^{\alpha}(-1)^{k+1}\frac{(n-k-1)!}{\alpha!(n-\alpha-1)!}\frac{p(\alpha,\alpha-k;\omega)}{(\omega-1)^{n-k}}. 
\end{equation}

\subsection{Definition of Szasz functions}
\label{DEFSF}

Let $C$ be a unimodular cone in $X$. 
Since the Riemann sum $R_{N}(C;\varphi)$ depends only on the restriction of $\varphi$ to $L(C)$, 
replacing $(X,\Lambda)$ by $(L(C),L(C) \cap \Lambda)$ if necessary, we assume, for a moment, 
that $\dim(C)=\dim(X)$. 
Then, $C$ is written in the form 
\[
C=\sum_{e \in E}\mb{R}_{+}e,  
\]
where $E$ is an integral basis of $\Lambda$, and $\mb{R}_{+}$ denotes the set of non-negative real numbers. 
For abstract two sets $S$ and $T$, let $S^{T}$ be the set of all functions from $T$ to $S$. 
The whole space $X$ is identified with $\mb{R}^{E}$. Since $E$ is an integral basis, $\Lambda$ is identified with $\mb{Z}^{E}$. 
Note that, $C$ and $C \cap \Lambda$ are identified with $\mb{R}_{+}^{E}$ and $\mb{Z}_{+}^{E}$, 
respectively, where $\mb{Z}_{+}$ denotes the set of non-negative integers. 
For any $\alpha \in \mb{Z}_{+}^{E}$ and $x \in X$, we set 
\[
\alpha!=\prod_{e \in E}\alpha(e)!,\quad x^{\alpha}=\prod_{e \in E}x(e)^{\alpha(e)}, 
\]
where $x(e)$ is the value of $x$ at $e \in E$ when we identify $X=\mb{R}^{E}$. 
For each $\gamma \in \mb{Z}_{+}^{E}$, we define the function $\ell_{\gamma}$ on $X$ by 
\begin{equation}
\label{coeffS}
\ell_{\gamma}(x)=\frac{x^{\gamma}}{\gamma!}e^{-\sum_{e \in E}x(e)}. 
\end{equation}
Then, the function $\ell_{\gamma}$ is non-negative, integrable on $C$ and satisfies
\begin{equation}
\label{ellP}
\int_{C}\ell_{\gamma}(x)\,dx=1\quad (\gamma \in \mb{Z}_{+}^{E}),\quad 
\sum_{\gamma \in \mb{Z}_{+}^{E}}\ell_{\gamma}(x)=1\quad (x \in X). 
\end{equation}

\begin{defin}
\label{SzDef}
We define the Szasz measure $\scal(x)=d\scal_{x}$ on $X=\mb{R}^{E}$, parametrized by $x \in X$, by 
\[
\scal(x)=d\scal_{x}:=\sum_{\gamma \in \mb{Z}_{+}^{E}}\ell_{\gamma}(x)\delta_{\gamma}. 
\]
\end{defin}
By the second property of \eqref{ellP}, the measure $d\scal_{x}$ is a probability measure on $C$. 
For each $N \in \mb{N}$, the $N$-th dilated convolution powers, denoted by $d\scal_{x}^{N}$, of $d\scal_{x}$ is given by 
\[
d\scal_{x}^{N}:=(D_{1/N})_{*}(\scal(x)\ast \cdots \ast \scal(x))=\sum_{\gamma \in \mb{Z}_{+}^{E}}\ell_{\gamma}(Nx)\delta_{\gamma/N},  
\]
where $D_{1/N}:X \to X$ is the dilation $D_{1/N}(x)=x/N$, $x \in X$. 
\begin{defin}
\label{SzDef2}
We define the Szasz function $S_{N}(\varphi)$ associated to a function $\varphi$ on $X$, by
\begin{equation}
\label{ds}
S_{N}(\varphi)(x):=\int_{C}\varphi(z)d\scal_{x}^{N}(z)=\sum_{\gamma \in \mb{Z}^{E}_{+}}
\ell_{\gamma}(Nx)\varphi(\gamma/N) 
\end{equation}
if the sum in the right hand side converges absolutely. 
\end{defin}
By \eqref{ellP}, the Szasz function $S_{N}(\varphi)$ satisfies that 
\begin{equation}
\label{SzaszRiem}
R_{N}(C;\varphi):=\frac{1}{N^{\dim (C)}}\sum_{\gamma \in C \cap \Lambda}\varphi(\gamma/N)=\int_{C}S_{N}(\varphi)(x)\,dx 
\end{equation}
if the sum converges absolutely. 

\subsection{Asymptotics of Szasz functions}
\label{AsymSF}

For each $\mu,\nu \in \mb{Z}_{+}^{E}$ with $\mu \geq \nu$, we define 
\begin{equation}
\label{mdimP}
p_{E}(\mu,\nu)=\prod_{e \in E}p(\mu(e),\nu(e)),  
\end{equation}
where $p(n,k)$ is an integer defined by \eqref{pnk}. 
For each $\mu \in \mb{Z}^{E}_{+}$, we set $\nabla^{\mu}=\prod_{e \in E}\nabla_{e}^{\mu(e)}$ and $|\mu|=\sum_{e \in E}\mu(e)$. 
Then, a relevant asymptotic formula for the Szasz function $S_{N}(\varphi)$ is given as follows.

\begin{prop}
\label{aSzasz}
For each positive integer $r$ and positive number $K$ with $r<K<2r$, 
there exists a positive constant $C_{r,K}$ such that we have 
\begin{equation}
\label{asymptS1}
S_{N}(\varphi)(x)=\sum_{\mu \in \mb{Z}^{E}_{+}\,;\,|\mu| \leq 2r-1}
\frac{\nabla^{\mu}\varphi(x)}{\mu!N^{|\mu|}}J_{\mu}(Nx) +S_{2r,N}(x), 
\end{equation}
where the function $S_{2r,N}(x)$ satisfies the following estimate; 
\begin{equation}
\label{error1}
|S_{2r,N}(x)| \leq C_{r,K}N^{-r}(1+|x|)^{r-K},\quad x \in C,   
\end{equation}
where the norm $|x|$ of $x \in X$ is defined by $|x|^{2}=\sum_{e \in E}x(e)^{2}$, 
and the function $J_{\mu}(x)$ is a polynomial in $x$ of degree at most $[|\mu|/2]$ given by 
\begin{equation}
\label{functJ}
J_{\mu}(x)=\sum_{\nu \in \mb{Z}_{+}^{E}\,;\,\nu \leq [\mu/2]}
p_{E}(\mu,\nu)x^{\nu}, 
\end{equation}
where $[\mu/2] \in \mb{Z}_{+}^{E}$ is defined by $[\mu/2](e)=[\mu(e)/2]$.  
\end{prop}

\begin{proof}
The proof is the same as that for Proposition \ref{tSzA}. 
Inserting the Taylor expansion 
\[
\begin{gathered}
\varphi(z)=\sum_{\mu \in \mb{Z}_{+}^{E}\,;\,|\mu| \leq 2r-1}\frac{\nabla^{\mu}\varphi(x)}{\mu!}(z-x)^{\mu}
+\sum_{|\mu|=2r}\frac{1}{\mu!}R_{2r,\mu}(z,x)(z-x)^{\mu}, \\
R_{2r,\mu}(z,x)=2r\int_{0}^{1}(1-t)^{2r-1}\nabla^{\mu}\varphi (x+t(z-x))\,dt
\end{gathered}
\]
with $z=\gamma/N$ into the definition \eqref{ds} of the Szasz function $S_{N}(\varphi)$, we have 
\[
S_{N}(\varphi)(x)=\sum_{\mu \in \mb{Z}_{+}^{E}\,;\,|\mu| \leq 2r-1}
\frac{\nabla^{\mu}\varphi (x)}{\mu!N^{|\mu|}}J_{\mu}(Nx)
+S_{2r,N}(x),
\]
where the functions $J_{\mu}(x)$, $S_{2r,\mu}(x)$ are given by 
\[
\begin{gathered}
J_{\mu}(x)=\sum_{\gamma \in \mb{Z}^{E}_{+}}\ell_{\gamma}(Nx)(\gamma-x)^{\mu},\\
S_{2r,N}(x)=\sum_{|\mu|=2r}\frac{1}{\mu!N^{|\mu|}}\sum_{\gamma \in \mb{Z}^{E}_{+}}
\ell_{\gamma}(Nx)R_{2r,\mu}(\gamma/N,x)(\gamma-Nx)^{\mu}. 
\end{gathered}
\]
The formula \eqref{functJ} is easily obtained by the relation 
\[
\sum_{\gamma \in \mb{Z}_{+}^{E}}\frac{\gamma^{\nu}}{\gamma!}x^{\gamma}=e^{\sum_{e \in E}x(e)}
\sum_{\alpha \leq \nu}S_{E}(\nu,\alpha)x^{\alpha},  
\]
which follows from Lemma \ref{spoly}, (1), where $S_{E}(\nu,\alpha)$ is given by 
\begin{equation}
\label{SE}
S_{E}(\nu,\alpha)=\prod_{e \in E}S(\nu(e),\alpha(e)). 
\end{equation}
Next, we estimate the term $S_{2r,N}(x)$. Note that $x$ and $\gamma/N$ are in $C$. 
Thus, we have $|x+t(\gamma/N-x)| \geq (1-t)|x|$ for each $0 \leq t \leq 1$. 
We choose a positive constant $C_{r,K}$ such that $|\nabla^{\mu}\varphi (x)| \leq C_{r,K}(1+|x|)^{-K}$ 
for each $\mu \in \mb{Z}_{+}^{E}$ with $|\mu|=2r$. 
Hence, if $|x| \geq 1$ and $|\mu|=2r$, we have 
\begin{equation}
\label{ss}
|\nabla^{\mu}\varphi(x+t(\gamma/N-x))|  \leq C_{r,K}(1+|x+t(\gamma/N-x)|)^{-K} 
 \leq C_{r,K}(1-t)^{-K}|x|^{-K} 
\end{equation}
Thus, for $|x| \geq 1$, 
we have $|R_{2r,\mu}(\gamma/N,x)| \leq C_{r,K}|x|^{-K} \leq C_{r,K}(1+|x|)^{-K}$,  
where $C_{r,K}$ is a constant. Therefore, we obtain 
\[
|S_{2r,N}(x)| \leq \frac{C_{r,K}}{N^{2r}}(1+|x|)^{-K}
\sum_{\gamma \in \mb{Z}^{E}_{+}}\ell_{\gamma}(Nx)\sum_{|\mu|=2r}\frac{1}{\mu!}|(\gamma -Nx)^{\mu}| 
\leq \frac{C_{r,K}}{N^{2r}}(1+|x|)^{-K}
\sum_{|\mu|=r}J_{2\mu}(Nx). 
\]
As is mentioned above, the function $J_{2\mu}(x)$ with $|\mu|=r$ is a polynomial in $x$ of degree at most  $r$. 
Thus, we have $|J_{\mu}(x)| \leq C_{\mu}|x|^{r}$ where $C_{\mu}$ does not depend on $x$. 
Therefore, we obtain the estimate \eqref{error1}. When $|x| \leq 1$, we estimate $R_{2r,\mu}(\gamma/N,x)$ as 
$|R_{2r,\mu}(\gamma/N,x)| \leq C_{r,K}$, and hence $S_{2r,N}(x)$ is bounded by $C_{r,K}N^{-r} \leq C_{r,K}N^{-r}(1+|x|)^{r-K}$, which 
completes the proof.
\end{proof}

\section{Asymptotic Euler-Maclaurin formula over unimodular cones}
\label{ASCONE}

In this section, we deduce asymptotic Euler-Maclaurin formula of the Riemann sum 
over unimodular cones in a rational space $(X,\Lambda)$. 
At first, we deduce it by using Proposition \ref{aSzasz}. 
The result coincide a well-known result due to Guillemin-Sternberg \cite{GS}. 
We don't need to use a rational inner product on $X$ so far. 
Then, we renormalize each term of the expansion using an integration by parts procedure 
to find explicit form of Berline-Vergne operators. 
This step involves a rational inner product.

\subsection{An Euler-Maclaurin formula for unimodular cones}
\label{GSEM}

As before, let $C$ be a unimodular cone in $X$ with $\dim(C)=\dim(X)$ and let $E$ be the integral basis of $\Lambda$ generating $C$. 
For each $I \subset E$, let $|I|$ be the number of elements in $I$. 
For such $I$, we regard $\mb{Z}_{+}^{I}$ as a subset of $\mb{Z}_{+}^{E}$ consisting of $\alpha \in \mb{Z}_{+}^{E}$ 
with the property that $\alpha(e)=0$ for each $e \in E \setminus I$. Clearly we have $\mb{Z}_{+}^{I} \subset \mb{Z}_{+}^{J}$ if $I \subset J$. 
For any $e \in E$, we define $\lambda_{e} \in \mb{Z}_{+}^{E}$ by 
$\lambda_{e}(e)=1$, $\lambda_{e}(v)=0$, $v \in E \setminus \{e\}$. 
Then, we obviously have $\alpha=\sum_{e \in E}\alpha(e)\lambda_{e}$ for each $\alpha \in \mb{Z}_{+}^{E}$. 
For $\emptyset \neq I \subset E$, we set $\mb{Z}_{>0}^{I}=\{\alpha \in \mb{Z}_{+}^{I}\,;\,\alpha(e) \neq 0,\,e \in I\}$. 
For $I=\emptyset$, we set $\mb{Z}_{>0}^{\emptyset}=\{0\}$. 
Each $I \subset E$ corresponds to a face $C(I)$ of $C$ defined by 
\begin{equation}
\label{facesC}
C(I):=\sum_{e \in E \setminus I}\mb{R}_{+}e, \quad C(E):=\{0\}, 
\end{equation}
and for each face $F$ of $C$, there is a unique $I_{F} \subset E$ such that $F=C(I_{F})$. 
Thus, we identify subsets in $E$ and faces of $C$. Note that $F \subset G$ if and only if $I_{G} \subset I_{F}$. 
For each $\mu,\nu \in \mb{Z}_{+}^{E}$ with $\nu \leq \mu$ and $\emptyset \neq I \subset E$, we set 
\begin{equation}
p_{I}(\mu,\nu):=\prod_{e \in I}p(\mu(e),\nu(e)). 
\end{equation}
If $\mu,\nu \in \mb{Z}_{+}^{I}$, we clearly have $p_{J}(\mu,\nu)=p_{I}(\mu,\nu)$ for each $J$ with $I \subset J$ because $p(0,0)=1$. 
For each $\nu \in \mb{Z}_{+}^{I}$, we set
\begin{equation}
\label{mdimPP}
p_{I}(\nu)=\sum_{\mu \in \mb{Z}_{+}^{I}\,;\,\nu \leq \mu \leq 2\nu}
(-1)^{|\mu|}\frac{(\mu-\nu)!}{\mu!}p_{I}(\mu,\mu-\nu). 
\end{equation}
Then, we have $p_{J}(\nu)=p_{I}(\nu)$ if $\nu \in \mb{Z}_{+}^{I}$ and $I \subset J$. 
Note that $p_{I}(\nu)=\prod_{e \in I}p(\nu(e))$, where we have $p(n)=(-1)^{n-1}c_{n}=(-1)^{n}b_{n}/n!$ as in \eqref{EMn2}, \eqref{Catalan}. 

For each non-negative integer $n$ and a subset $I$ of $E$ with $|I| \leq n$, we define a homogeneous 
differential operator $L_{n}(C;I)$ of order $n-|I|$ on $X$ with constant coefficients by 
\begin{equation}
\label{diffop1}
L_{n}(C;I)=(-1)^{n}\sum_{\nu \in \mb{Z}_{>0}^{I}\,;\,|\nu|=n}p_{I}(\nu)\nabla^{\nu-e(I)}, \qquad
e(I)=\sum_{e \in I}\lambda_{e}\quad (n \geq 1), 
\end{equation}
and $L_{0}(C;\emptyset)=1$. 
When $n \geq 1$ we set $L_{n}(C;I)=0$ for $|I| >n$ or $I=\emptyset$. 

\begin{prop}
\label{uC}
For each $\varphi \in \scal(X)$, we have 
\begin{equation}
\label{aRiemC1}
R_{N}(C;\varphi) \sim \sum_{n \geq 0}N^{-n}\sum_{I \subset E\,;\,|I| \leq n}
(-1)^{|I|}\int_{C(I)}
L_{n}(C;I)\varphi. 
\end{equation}
\end{prop}

\begin{proof}
We use Proposition \ref{aSzasz}. We take $r \in \mb{N}$ and $K>0$ so that $r+\dim(X) <K<2r$. 
By the estimate \eqref{error1}, one can integrate the asymptotic expansion \eqref{asymptS1} over $C$. 
Then, by \eqref{SzaszRiem} and \eqref{asymptS1}, we have 
\[
R_{N}(C;\varphi) =
\sum_{\mu,\nu\,;\,|\mu| \leq 2r-1,\nu \leq [\mu/2]}
\frac{1}{\mu!}N^{-|\mu-\nu|}p_{E}(\mu,\nu)\int_{C}x^{\nu}\nabla^{\mu}\varphi +O(N^{-r}).
\]
Integrating by parts, we have $\displaystyle \int_{C}x^{\nu}\nabla^{\mu}\varphi =(-1)^{|\nu|}\nu!\int_{C}\nabla^{\mu-\nu}\varphi$, 
and hence, substituting this into the formula for $R_{N}(C;\varphi)$ above, we obtain
\begin{equation}
\label{As0}
\begin{split}
R_{N}(C;\varphi)&=\sum_{k=0}^{r-1}
N^{-k}\int_{C}L_{k}(C)\varphi +O(N^{-r}), \\
L_{k}(C)&=(-1)^{k}\sum_{\nu \in \mb{Z}_{+}^{E},\,|\nu|=k}p_{E}(\nu)\nabla^{\nu} 
=\sum_{\nu,\mu,\,|\nu|=k,\,\nu \leq \mu \leq 2\nu}(-1)^{|\mu|+k}\frac{(\mu-\nu)!}{\mu!}p_{E}(\mu,\mu-\nu)\nabla^{\nu}. 
\end{split}
\end{equation}
To integrate by parts further in the right hand side, note that 
we have $\mb{Z}_{+}^{E}=\bigcup_{I \subset E}\mb{Z}_{>0}^{I}$, 
which is a disjoint union. For $\nu \in \mb{Z}_{>0}^{I}$, we have 
\[
\int_{C}\nabla^{\nu}\varphi=(-1)^{|I|}\int_{C(I)}\nabla^{\nu-e(I)}\varphi. 
\]
From this, we obtain
\[
\int_{C}L_{k}(C)\varphi=
\sum_{I \subset E\,;\,|I| \leq k}(-1)^{|I|}
\int_{C(I)}L_{k}(C;I)\varphi, 
\]
which shows the assertion. 
\end{proof}

Next, we consider cones containing straight lines. Let $E$ be an integral basis of $\Lambda$, 
and let $I \subset E$.  
Consider the cone $C$ in $X$ of the form
\begin{equation}
\label{genC}
C=\sum_{e \in I}\mb{R}_{+}e+L, 
\end{equation}
where $L$ is a subspace in $X$ spanned by vectors $e \in E \setminus I$. 
If $I=E$, then $L=\{0\}$ and in this case $C$ is a unimodular cone discussed above. 
When $I =\emptyset$, we set $C=X$. 

\begin{lem}
\label{notuC}
Let $E$ be an integral basis of $\Lambda$, and let $C$ be a cone of the form $\eqref{genC}$ with $I \subset E$. 
Then, for each $\varphi \in C_{0}^{\infty}(X)$, we have 
\begin{equation}
\label{NSCas}
R_{N}(C;\varphi)=R_{N}(\pi_{L}(C);(\pi_{L})_{*}\varphi)+O(N^{-\infty}), 
\end{equation}
where $(\pi_{L})_{*}\varphi$ is a compactly supported smooth function on $X/L$ defined by 
\[
(\pi_{L})_{*}\varphi (x)=\int_{\pi_{L}^{-1}(x)}\varphi,\quad x \in X/L. 
\]
\end{lem}

\begin{proof}
For simplicity, we write $\pi=\pi_{L}:X \to X/L$ for the natural projection. 
Take $\varphi \in C_{0}^{\infty}(X)$. For any $v \in X$, we set $T_{v}\varphi (x) :=\varphi(x+v)$. 
Let $M$ be the subspace spanned by $I$ so that $X=M \oplus L$. 
We identify $L$ with $\mb{R}^{E \setminus I}$ and $M$ with $\mb{R}^{I}$ in a natural way. 
Then, we can choose $v \in \mb{Z}^{E\setminus I}$ 
so that $\supp(T_{v}\varphi) \subset M +\mb{R}_{>0}^{E \setminus I}$. 
Clearly we have $R_{N}(C;\varphi)=R_{N}(\mb{R}_{+}^{E};T_{v}\varphi)$, where we note that 
$\mb{R}_{+}^{E}$ is a unimodular cone in $X$. Therefore, by \eqref{As0}, we have 
\[
R_{N}(C;\varphi) \sim \sum_{n \geq 0}N^{-n}\int_{\mb{R}_{+}^{E}}L_{n}(\mb{R}_{+}^{E})T_{v}\varphi,
\]
where the differential operator $L_{n}(\mb{R}_{+}^{E})$ is given in \eqref{As0}. 
Note that $\pi(C)$ is a unimodular cone in $X/L$ with respect to the lattice $\pi(\Lambda)$ generated 
by the integral basis $\pi(I)$ of $\pi(\Lambda)$. 
Since $v \in \mb{Z}^{E \setminus I}$, we have $\pi_{*}T_{v}\varphi=\pi_{*}\varphi$. 
Therefore, according to Proposition \ref{uC}, we only need to show that 
\begin{equation}
\label{reduced1}
\int_{\mb{R}_{+}^{E}}L_{n}(\mb{R}_{+}^{E})\psi\,dx
=\int_{\pi(C)}L_{n}(\pi(C))\pi_{*}\psi\,dx 
\end{equation}
for any $\psi \in C_{0}^{\infty}(X)$ with $\supp(\psi) \subset M+\mb{R}_{>0}^{E \setminus I}$. 
If $\nu \in \mb{Z}_{+}^{E}$ has some $e \in E \setminus I$ such that $\nu(e) \geq 1$, then, 
since $\supp(\psi) \cap \mb{R}_{+}^{E \setminus \{e\}}=\emptyset$, we have $\displaystyle \int_{\mb{R}_{+}^{E}}\nabla^{\nu}\psi=0$ 
and hence 
\[
\int_{\mb{R}_{+}^{E}}L_{n}(\mb{R}_{+}^{E})\psi\,dx
=\int_{\mb{R}_{+}^{E}}\tilde{L}_{n}\psi,\quad 
\tilde{L}_{n}=(-1)^{n}\sum_{\nu \in \mb{Z}_{+}^{I}\,;\,|\nu|=n}p_{I}(\nu)\nabla^{\nu}. 
\]
If we denote $\nabla^{\nu}_{\pi}=\prod_{e \in I}\nabla_{\pi(e)}^{\nu(e)}$ for each $\nu \in I$, 
then, by the definition of the function $\pi_{*}\psi$ on $X/L$, we have $\nabla^{\nu}_{\pi}\pi_{*}\psi=\pi_{*}\nabla^{\nu}\psi$ for each $\nu \in \mb{Z}_{+}^{I}$. 
Since $\supp(\psi) \subset M+\mb{R}_{+}^{E \setminus I}$, we obtain, for $\nu \in \mb{Z}_{+}^{I}$, 
\[
\int_{\pi(C)}\nabla_{\pi}^{\nu}\pi_{*}\psi
=\int_{\pi(C)}\pi_{*}\nabla^{\nu}\psi
=\int_{\mb{R}_{+}^{E}}\nabla^{\nu}\psi.  
\]
From this and the definition of $L_{n}(\pi(C))$, we obtain \eqref{reduced1}. 
\end{proof}

\begin{rem}
As is mentioned in Introduction, Proposition \ref{uC} is deduced directly from Theorem 3.2 in \cite{GS}. 
Lemma \ref{notuC} is also obtained in \cite{GS}. 
\end{rem}

\subsection{Integration by parts}
\label{IBP}

In Proposition \ref{uC} and Lemma \ref{notuC} in the previous subsection, 
we have derived an asymptotic formula for the Riemann sums over unimodular cones and their variants. 
In each term in these asymptotic formulas, integration over faces of homogeneous differential operators $L_{n}(C;I)$ 
defined in \eqref{As0} appears. The differential operators $L_{n}(C;I)$ involve derivatives only in directions transversal to the face $C(I)$. 
However, these derivatives are not `perpendicular' to the face $C(I)$, and hence we can perform further integration by parts. 
If one performs integration by parts in \eqref{aRiemC1}, then one will find the differential operators 
which involves derivatives only in directions perpendicular to faces. However, we need to perform this procedure 
systematically to define the operators all at once.  
This step is one of the main points which makes the final formula complicated.

In the rest of this paper, we fix a rational inner product $Q$ on the rational space $(X,\Lambda)$. 
Let $E$ be an integral basis of $\Lambda$. 
For each $I \subset E$, we set 
\begin{equation}
\label{ratsubsp}
X(I)=\bigoplus_{e \in E \setminus I}\mb{R}e \cong \mb{R}^{E \setminus I},\quad 
X(E)=\{0\}.   
\end{equation}
Note that $X=X(\emptyset)$, and if $I \subset J$, then $X(J) \subset X(I)$ and hence $X(I)^{\perp_{Q}} \subset X(J)^{\perp_{Q}}$. 
As before, for each $I \subset E$, we define the unimodular cone $C(I)$ in $X$ by \eqref{facesC}. 
For each $\alpha \in \mb{Z}_{+}^{E}$, we set $\nabla^{\alpha}=\prod_{e \in E}\nabla_{e}^{\alpha(e)}$.

\begin{prop}
\label{Uibp}
There exists a family 
\[
\{L(E;I,J;\alpha)\,;\,\emptyset \neq I \subset J \subset E, \alpha \in \mb{Z}_{+}^{I}, |J| \leq |\alpha| +|I|\}
\]
of homogeneous differential operators $L(E;I,J;\alpha)$ of order $|\alpha|-|J|+|I|$ 
on $X$ with rational constant coefficients which involves derivatives only in directions perpendicular to the 
rational subspace $X(J)$ such that for each $(I,\alpha)$ with $\emptyset \neq I \subset E$, $\alpha \in \mb{Z}_{+}^{I}$, we have 
\begin{equation}
\label{intbyp1}
\int_{C(I)}\nabla^{\alpha}\varphi
=\sum_{J\,;\,I \subset J,\,|J| \leq |\alpha|+|I|}
(-1)^{|J|-|I|}
\int_{C(J)}L(E;I,J;\alpha)\varphi
\end{equation}
for any $\varphi \in \scal(X)$. 
Furthermore, fix $\alpha$ and $\emptyset \neq I \subset E$ with $\alpha \in \mb{Z}_{+}^{I}$.  
Suppose that a family $\{L(J)\,;\,I \subset J \subset E, |J| \leq |\alpha| + |I|\}$ of 
homogeneous differential operators with constant coefficients of order $|\alpha| - |J| + |I|$ 
which involves derivatives only in directions 
perpendicular to $X(J)$ satisfy the equation $\eqref{intbyp1}$ for any $\varphi \in \scal(X)$. 
Then, we have $L(J)=L(E;I,J;\alpha)$. 
\end{prop}

Note that a differential operator on $X$ is said to have rational coefficients 
if it has rational coefficients with respect to an (and hence all) integral basis of $\Lambda$. 
We first prove the existence of such family of differential operators. 

\vspace{10pt}

\noindent{{\it Proof of the existence in {\rm Proposition \ref{Uibp}}.}}\hspace{5pt}
For a given $\emptyset \neq I \subset E$ and $e \in I$, we decompose $e$ along with the 
orthogonal decomposition $X=X(I)^{\perp_{Q}} \oplus X(I)$, which is denoted by 
\begin{equation} 
\label{decoI}
e=u(I;e)+\sum_{v \in E \setminus I}c(I;e,v)v,\quad u(I;e) \in X(I)^{\perp_{Q}} \cap X_{\mb{Q}},\quad c(I;e,v) \in \mb{Q},\quad e \in I. 
\end{equation}
We set $u(E;e)=e$ for each $e \in E$.
We construct the operators $L(E;I,J;\alpha)$ inductively as follows. 

\vspace{5pt}

\noindent{(0)} When $|\alpha|=0$, then $|J|=|I|$ and $I \subset J$ implies $I=J$. In this case, we set 
\begin{equation}
\label{opdef0}
L(E;I,I;0)=1. 
\end{equation}

\vspace{5pt}

\noindent{(1)} We take $\emptyset \neq I \subset J \subset E$ and $\alpha \in \mb{Z}_{+}^{I}$ with $|\alpha|=1$ and 
$|J| \leq |\alpha| +|I|$. In this case, $J=I$ or $J=I \cup \{v\}$ with $v \in E \setminus I$ and $\alpha=\lambda_{e}$ with $e \in I$. 
We then define $L(E;I,J;\lambda_{e})$ by 
\begin{equation}
\label{opdef1}
\left\{
\begin{array}{ll}
L(E;I,I;\lambda_{e})=\nabla_{u(I;e)} & (\mbox{when $I=J$}), \\
L(E;I,I \cup\{v\};\lambda_{e})=c(I;e,v) & (\mbox{when $J=I \cup\{v\}$ with $v \in E \setminus I$}). 
\end{array}
\right.
\end{equation}
Then, by the identity
\[
\int_{C(I)}\nabla_{v}\varphi=-\int_{C(I \cup \{v\})}\varphi,\quad v \in E \setminus I,\quad \varphi \in C_{0}^{\infty}(X), 
\]
it is easy to show that the operators $L(E;I,I \cup \{v\};\lambda_{e})$ satisfy
\begin{equation}
\label{intbyp0}
\int_{C(I)}\nabla_{e}\varphi =\int_{C(I)}L(E;I,I;\lambda_{e})\varphi +\sum_{v \in E \setminus I}(-1)\int_{C(I \cup \{v\})}L(E;I,I \cup \{v\};\lambda_{e})\varphi. 
\end{equation}

\vspace{5pt}

\noindent{(2)} Suppose that, for a positive integer $n \geq 2$, 
we have defined differential operators $L(E;I,J;\beta)$ satisfying the formula \eqref{intbyp1}
for any $\varphi \in C_{0}^{\infty}(X)$ for each $I$, $J$, $\beta$ with 
$\emptyset \neq I \subset J \subset E$, $\beta \in \mb{Z}_{+}^{I}$ satisfying $|\beta| \leq n-1$, $|J| \leq |\beta|+|I|$.

\vspace{5pt}

\noindent{(3)} For $\emptyset \neq I \subset J \subset E$ and $\alpha \in \mb{Z}_{+}^{I}$ satisfying 
$|\alpha|=n$ and $|J| \leq n+|I|$, we take $e \in I$ such that $\alpha(e) \geq 1$. Then, 
we can decompose $\alpha$ as 
\begin{equation}
\label{decoA}
\alpha=\lambda_{e}+\beta,\quad \beta \in \mb{Z}_{+}^{I}, \quad |\beta|=n-1. 
\end{equation}
We define $L(E;I,J;\alpha)$ by the formula 
\begin{equation}
\label{opdef2}
L(E;I,J;\alpha)=\left\{
\begin{array}{l}
L(E;I,I;\beta)\nabla_{u(I;e)} 
\hspace{20pt} (\mbox{when $J=I$}), \\
L(E;I,J;\beta)\nabla_{u(I;e)}+\sum_{v \in J \setminus I}c(I;e,v)L(E;I\cup\{v\},J;\beta) \\
\hspace{20pt} (\mbox{when $|I|+1 \leq |J| \leq |I|+|\alpha|-1$}), \\
\sum_{v \in J \setminus I}c(I;e,v)L(E;I\cup\{v\},J;\beta) 
\hspace{20pt} (\mbox{when $|J|=|I|+|\alpha|$}). 
\end{array}
\right.
\end{equation}
A direct computation shows that the differential operators $L(E;I,J;\alpha)$ satisfy \eqref{intbyp1}. 
\hfill$\square$

\vspace{10pt}

Next, we proceed to prove the uniqueness of such family $\{L(E;I,J;\alpha)\}$. 
Fix $\alpha$ and $\emptyset \neq I \subset E$ such that $\alpha \in \mb{Z}_{+}^{I}$. 
If $\alpha=0$, then clearly the operator $L(E;I,I;0)$ satisfying \eqref{intbyp1} is just the constant $1$. 
If $I=E$, then for any $\alpha \in \mb{Z}_{+}^{E}$, the operator $L(E;E,E;\alpha)$ is uniquely determined as $L(E;E,E;\alpha)=\nabla^{\alpha}$. 
Thus, in the following, we assume $\alpha \neq 0$ and $\emptyset \neq I \subsetneq E$. 

If the homogeneous differential operators $\{L(J)\}_{I \subset J \subset E\,;\,|J| \leq |\alpha|+|I|}$ with constant coefficients 
of order $|\alpha|-|J|+|I|$ which involves derivatives only in directions perpendicular to $X(J)$ satisfy the 
equation \eqref{intbyp1} for any $\varphi \in \scal(X)$, then their symbols $\sigma(L(J))$ must satisfy the equation 
\begin{equation}
\label{symbol1}
\begin{split}
&\xi^{\alpha}=\sum_{J\,;\,I \subset J,|J| \leq |\alpha|+|I|}
\sigma(L(J))(\xi)\prod_{e \in J \setminus I}\ispa{\xi,e},\qquad 
\xi^{\alpha}=\prod_{e \in E}\ispa{\xi,e}^{\alpha(e)}, \\
&\sigma(L(J))(\xi)=\sigma(L(J))(p_{J}(\xi)),\qquad \xi \in X^{*}, 
\end{split}
\end{equation}
where the symbol $\sigma(D)$ of a differential operator $D$ on $X$ (with constant coefficients) is 
a polynomial function on $X^{*}$ characterized by $\sigma(D)(\xi)=e_{-\xi}De_{\xi}$, $e_{\xi}(x)=e^{\ispa{\xi,x}}$, $x \in X$, $\xi \in X^{*}$. 
In \eqref{symbol1}, $p_{J}$ denotes the orthogonal projection from $X^{*}$ onto 
the annihilator $X(J)^{\perp}$ of $X(J)$. 
Therefore, to prove the uniqueness in the statement of Proposition \ref{Uibp}, 
it is enough to show the uniqueness of the family of homogeneous polynomials $\{\sigma(L(J))\}$ satisfying \eqref{symbol1}. 
First of all, consider the following expression. 
\begin{equation}
\label{symbolS0}
\sigma(I,I;\alpha)(\xi)=p_{I}(\xi)^{\alpha},  
\end{equation}
\begin{equation}
\label{symbolS1}
\begin{gathered}
\sigma(I,J;\alpha)(\xi)=\frac{p_{J}(\xi)^{\alpha}-\sum_{i=0}^{k-1}\sigma_{i}(I;\alpha)(p_{J}(\xi))}{\prod_{e \in J \setminus I}\ispa{p_{J}(\xi),e}}, 
\quad I \subset J, |J|=k+|I|,\\
\sigma_{i}(I;\alpha)(\xi)=\sum_{J\,;\,I \subset J,|J|=|I|+i}
\sigma(I,J;\alpha)(\xi)\prod_{e \in J \setminus I}\ispa{\xi,e},  
\end{gathered}
\end{equation}
where $k$ is an integer satisfying $1 \leq k \leq |\alpha|$. 
Note that $\sigma_{0}(I;\alpha)=\sigma(I,I;\alpha)=p_{I}(\xi)^{\alpha}$ is a well-defined homogeneous 
polynomial of degree $|\alpha|$ on $X^{*}$. 
Thus, the above equations \eqref{symbolS0}, \eqref{symbolS1} define rational functions $\sigma(I,J;\alpha)$, $\sigma_{i}(I;\alpha)$ 
for $\emptyset \neq I \subset J \subset E$, $\alpha \in \mb{Z}_{+}^{I}$, $|J| \leq |\alpha|+|I|$, $0 \leq i \leq |\alpha|$, which are homogeneous 
of degree $|\alpha|-|J|+|I|$, $|\alpha|$, respectively. 
Note also that the functions $\sigma(I,J;\alpha)$ satisfy the second line of \eqref{symbol1}.

\begin{lem}
\label{polyL}
The functions defined by $\eqref{symbolS0}$, $\eqref{symbolS1}$ are homogeneous polynomials. 
\end{lem}

\begin{proof}
First of all, let us examine the function $\sigma(I,J;\alpha)$ with $I \subset J$, $|J|=1+|I|$. 
In this case, we can write $J=I \cup \{u\}$ with some $u \in E \setminus I$. 
By $\eqref{symbolS1}$, we have 
\[
\sigma(I,I \cup \{u\};\alpha)(\xi)=\frac{p_{I \cup \{u\}}(\xi)^{\alpha}-p_{I}(\xi)^{\alpha}}{\ispa{p_{I \cup \{u\}}(\xi), u}}. 
\]
Take $\xi \in X(I \cup \{u\})^{\perp}$, which means that $\ispa{\xi,e}=0$ for each $e \in E \setminus I$, $e \neq u$. 
Thus, there exists an $\eta \in X(I \cup \{u\})^{\perp}$ perpendicular to $X(I)^{\perp}$ with respect to the inner product $Q$ 
such that $Q(\eta,\eta)=1$. Note that $\ispa{\eta,u} \neq 0$. 
Let $q(\xi)=Q(\xi,\eta)\eta$ denote the orthogonal projection onto the one-dimensional subspace $\mb{R}\eta$. 
Then, we have $p_{I \cup \{u\}}=p_{I} +q$, and hence 
\[
\sigma(I,I \cup\{u\};\alpha)(\xi)=
\frac{(p_{I}(\xi)+q(\xi))^{\alpha}-p_{I}(\xi)^{\alpha}}{\ispa{q(\xi),u}}, 
\]
which is a homogeneous polynomial of degree $|\alpha|-1$. Next, to use the induction, suppose that 
the functions $\sigma(I,J;\alpha)$ with $I \subset J$, $|J| \leq k+|I|$ ($1 \leq k \leq |\alpha|-1$) are 
homogeneous polynomials of degree $|\alpha|-|J|+|I|$. By \eqref{symbolS1}, the functions $\sigma_{i}(I;\alpha)$ ($i=0,\ldots,k$) are 
homogeneous polynomials of degree $|\alpha|$. 
Take $J_{0} \subset E$ such that $I \subset J_{0}$, $|J_{0}|=k+1+|I|$. 
Set $f(\xi)=p_{J_{0}}(\xi)^{\alpha}-\sum_{i=1}^{k}\sigma_{i}(I;\alpha)(p_{J_{0}}(\xi))$. 
Note that the polynomial $f$ is determined on $X(J_{0})^{\perp}$. So, let $\xi \in X(J_{0})^{\perp}$. 
Assume that $\ispa{\xi,e}=0$ for some $e \in J_{0} \setminus I$, which means that $\xi \in X(K)^{\perp}$ with $K=J_{0} \setminus \{u\}$. 
Note that $I \subset K \subsetneq J$, $|K|=k+|I|$. Since $\xi \in X(K)^{\perp} \subset X(J_{0})^{\perp}$, we have $p_{J_{0}}(\xi)=\xi$, and hence 
by $\eqref{symbolS1}$, 
\[
\sigma_{k}(I;\alpha)(\xi)=\sigma(I,K;\alpha)(\xi)\prod_{e \in K \setminus I}\ispa{\xi,e}
=\xi^{\alpha}-\sum_{i=0}^{k-1}\sigma_{i}(I;\alpha)(\xi). 
\]
This shows that $f(\xi)=0$ for $\xi \in X(K)^{\perp}$. Thus, the homogeneous polynomial $f(\xi)$ is divisible by the linear function $\ispa{p_{J_{0}}(\xi),e}$ for 
each $e \in J_{0} \setminus I$. Note that the elements in $J_{0} \setminus I$ are linearly independent. 
Therefore, the homogeneous polynomial $f(\xi)$ is divisible by $\prod_{e \in J_{0} \setminus I}\ispa{p_{J_{0}}(\xi),e}$, and hence $\sigma(I,J_{0};\alpha)$ is a 
homogeneous polynomial. This completes the proof. 
\end{proof}

\vspace{5pt}

\noindent{{\it Proof of the uniqueness in} Proposition \ref{Uibp}.}\hspace{5pt}
Fix $\alpha$, $I$ such that $0 \neq \alpha \in \mb{Z}_{+}^{I}$, $\emptyset \neq I \subsetneq E$. 
Suppose that the set of functions $\{s(J)\}_{I \subset J \subset E,|J| \leq |\alpha|+|I|}$, where 
each $s(J)$ is a homogeneous function on $X^{*}$ of degree $|\alpha|-|J|+|I|$, satisfy 
the equation \eqref{symbol1}. 
Let $\sigma(I,J;\alpha)$ denote the homogeneous polynomials defined by \eqref{symbolS0}, \eqref{symbolS1}. 
We need to prove $s(J)=\sigma(I,J;\alpha)$. 
Let $\xi \in X(I)^{\perp}$. Then, we have $\ispa{\xi,e}=0$ for each $e \in E \setminus I$. Thus, by $\eqref{symbol1}$, 
we have $s(I)(\xi)=\sigma(I,I;\alpha)(\xi)$ for each $\xi \in X(I)^{\perp}$. Since $s(I)$ satisfies the 
second equation of \eqref{symbol1}, we have $s(I)=\sigma(I,I;\alpha)$ on $X^{*}$. 
Next, take $J_{0} \subset E$ such that $I \subset J_{0}$, $|J_{0}|=1+|I| \leq |\alpha|+|I|$. 
We write $J_{0}=I \cup \{u\}$. Take $\xi \in X(J_{0})^{\perp}$. 
Then $\ispa{\xi,e}=0$ for each $e \in E \setminus J_{0}$, and hence the function $s(J_{0})$ must satisfy
\[
\xi^{\alpha}=\sigma_{0}(I;\alpha)(\xi)+s(J_{0})(\xi)\ispa{\xi,u},\quad \xi \in X(J_{0})^{\perp}. 
\]
Since the function $s(J_{0})$ satisfies the second line of \eqref{symbol1}, we have
\[
s(J_{0})(\xi)=
\frac{(p_{I}(\xi)+q(\xi))^{\alpha}-p_{I}(\xi)^{\alpha}}{\ispa{q(\xi),u}},\quad J_{0}=I \cup \{u\}, 
\]
where $q$ is the orthogonal projection onto the one-dimensional subspace of $X(J_{0})^{\perp}$ perpendicular to $X(I)^{\perp}$. 
This equation shows $s(J_{0})=\sigma(I,J_{0};\alpha)$ for $J_{0}=I \cup \{u\}$. 
Now, suppose that for any $J \subset E$ with $I \subsetneq J$, $|J| \leq k+|I|$, $k+1 \leq |\alpha|$, 
we have $s(J)=\sigma(I,J;\alpha)$. 
Take $J_{0} \subset E$ with $I \subset J_{0}$, $|J_{0}|=k+1+|I| \leq |\alpha|+|I|$. Take $\xi \in X(J_{0})^{\perp}$. 
Since $\ispa{\xi,e}=0$ for each $e \in E \setminus J_{0}$, the equation \eqref{symbol1} shows
\begin{equation}
\label{symbol2}
\xi^{\alpha}=s_{k}(\xi)+s(J_{0})(\xi)\prod_{e \in J_{0} \setminus I}\ispa{\xi,e}, \quad 
s_{k}(\xi)=\sum_{I \subset J \subsetneq J_{0},\,|J| \leq k+|I|}\sigma(I,J;\alpha)(\xi)
\prod_{e \in J \setminus I}\ispa{\xi,e}, 
\end{equation}
for each $\xi \in X(J_{0})^{\perp}$. If $J \subset E$ with $I \subset J$ satisfy $J \setminus J_{0} \neq \emptyset$, 
then we have $\prod_{e \in J\setminus I}\ispa{\xi,e}=0$ and hence $s_{k}(\xi)=\sum_{i=0}^{k}\sigma_{i}(I;\alpha)(\xi)$. 
Since $s(J_{0})$ satisfies the second line of \eqref{symbol1}, we have $s(J_{0})=\sigma(I,J_{0};\alpha)$. 
\hfill$\square$

\vspace{10pt}

\begin{rem}
One can prove directly that the polynomials $\sigma(I,J;\alpha)$ defined by \eqref{symbolS0}, \eqref{symbolS1} actually a solution to \eqref{symbol1}. 
However, its proof is similar to the proof of the existence in Proposition \ref{Uibp}, and hence we omit it. 
\end{rem}

In the above, we have defined the differential operators $L(E;I,J;\alpha)$ for each $\emptyset \neq I \subset J \subset E$ and 
$\alpha \in \mb{Z}_{+}^{I}$ satisfying $|J| \leq |\alpha|+|I|$. But we need to work on the quotient space 
$X/L$ and the unimodular cone $\pi(C)$ which is the image of a unimodular cone $C$ under 
the natural projection $\pi:X \to X/L$ where $L$ is a subspace spanned by a subset of $E$. 
To state the next lemma, we need to fix some notation. 
Let $E$ be an integral basis of $\Lambda$. 
For $\emptyset \neq K \subset E$, we set, as before, $X(K)=\bigoplus_{v \in E \setminus K}\mb{R}v$. 
Let $\pi_{K}:X \to X/X(K)$ be the natural projection. For each $e \in E$, we set $\ol{e}=\pi_{K}(e)$. 
Then, we have $\pi_{K}(E)=\pi_{K}(K)=\{\ol{e}\,;\,e \in K\}$, and the set $\pi_{K}(K)$ is an integral basis of the lattice $\pi_{K}(\Lambda)$ in $X/X(K)$. 
Note that $\pi_{K}$ is a bijective map from $K$ onto $\pi(E)$. 
For each $\emptyset \neq I \subset K$ and $\alpha \in \mb{Z}_{+}^{I}$, 
denote $\pi_{K}(\alpha) \in \mb{Z}_{+}^{\pi(I)}$ the $\mb{Z}_{+}$-valued function on $\pi_{K}(I)$ defined by 
\begin{equation}
\label{quotientA}
\pi_{K}(\alpha)(\ol{e}):=\alpha(e),\quad e \in I. 
\end{equation}
We note that, for each $\ol{\alpha} \in \mb{Z}_{+}^{\pi(I)}$, there is a unique $\alpha \in \mb{Z}_{+}^{I}$ with 
the property that $\pi_{K}(\alpha)=\ol{\alpha}$. 

\begin{lem}
\label{UD}
In the notation as above, we identify $X/X(K)$ with $X(K)^{\perp_{Q}}$ 
to give $X/X(K)$ the inner product induced by the inner product $Q$ on $X$. 
Then, for each $\emptyset \neq I \subset J \subset K$ and $\alpha \in \mb{Z}_{+}^{I}$ with $|J| \leq |\alpha| +|I|$, we have 
\begin{equation}
\label{updown1}
L(E;I,J;\alpha)=L(\pi_{K}(K);\pi_{K}(I),\pi_{K}(J);\pi_{K}(\alpha))
\end{equation}
where the operator $L(\pi_{K}(K);\pi_{K}(I),\pi_{K}(J);\pi_{K}(\alpha))$ is 
regarded as an operator on $X$ by the identification $X/X(K) \cong X(K)^{\perp_{Q}}$. 
\end{lem}

\begin{proof}
For simplicity, we set $\pi=\pi_{K}$. 
Let $\sigma_{K}(I,J;\alpha)$ denote the symbol of the differential operator $L(\pi(K);\pi(I),\pi(J);\pi(\alpha))$ 
which is a homogeneous polynomial on $(X/X(K))^{*}$, and note that we have the identification $(X/X(K))^{*} \cong X(K)^{\perp}$ 
under the transpose $\!\,^{t}\pi:(X/X(K))^{*} \to X^{*}$ of $\pi$. 
Then, the symbol of the lift of the differential operator $L(\pi(K);\pi(I),\pi(J);\pi(\alpha))$ is given by 
$\sigma_{K}(I,J;\alpha)(p_{K}(\xi))$, $\xi \in X^{*}$. The symbols of the operators $L(E;I,J;\alpha)$ for $J \subset K$, which 
are as above denoted by $\sigma(I,J;\alpha)$, are determined on $X(K)^{\perp}$. 
By \eqref{symbol1}, we have 
\begin{equation}
\label{symbolQ}
\xi^{\alpha}=\sum_{\stackrel{J\,;\,I \subset J \subset K}{|J| \leq |\alpha| + |I|}}
\sigma(I,J;\alpha)(\xi)\prod_{e \in J \setminus I}\ispa{\xi,e},\quad \xi \in X(K)^{\perp}. 
\end{equation}
In Proposition \ref{Uibp}, we can replace $X$ by $X(E \setminus K)$ which is identified, as a rational space, with $X/X(K)$. 
With this identification, the symbols $\sigma_{K}(I,J;\alpha)$ also satisfy the equation \eqref{symbolQ}. 
Noting that the equation \eqref{symbolQ} is nothing but the equation \eqref{symbol1} on $X(E \setminus K)$, 
and using the uniqueness in Proposition \ref{Uibp}, we conclude the assertion. 
\end{proof}

\subsection{Berline-Vergne operators over unimodular cones}
\label{DOC}

We use the results obtained in the previous subsections and Theorem \ref{BerVer} to find 
an explicit expression of Berline-Vergne operators for unimodular cones.

\begin{defin}
\label{opcone}
\noindent{\rm (1)}\hspace{5pt} Let $C$ be a unimodular cone in a rational space $(X,\Lambda)$ with 
a rational inner product $Q$. 
Assume that $\dim(C)=\dim(X)$. Let $E$ be the integral basis of $\Lambda$ generating $C$. 
For each $F \in \fcal(C)$, we take, as before, a unique $I_{F} \subset E$ such that $F=C(I_{F})$. 
Then, for each $F \in \fcal(C)$ and $n \in \mb{Z}_{+}$ with $\dim(F) \geq \dim(C)-n$, 
we define a homogeneous differential operator $\dcal_{n}^{X}(C;F)$ 
of order $n-\dim(C)+\dim(F)$ with rational constant coefficients 
which involves derivatives only in directions perpendicular to the face $F$ by 
\begin{equation}
\label{ConeD}
\dcal_{n}^{X}(C;F)
:=(-1)^{n-\dim(C)+\dim (F)}\sum_{I \subset I_{F}}
\sum_{\nu \in \mb{Z}_{>0}^{I},\,|\nu|=n}
p_{I}(\nu)L(E;I,I_{F};\nu-e(I)),  
\end{equation}
and $\dcal_{0}^{X}(C;C):=1$, $\dcal_{n}^{X}(C;C):=0$ $(n \geq 1)$. 

\vspace{5pt}

\noindent{\rm (2)}\hspace{5pt} Let $C \subset X$ be a unimodular cone. 
For any $F \in \fcal(C)$ and $n \in \mb{Z}_{+}$ with $n-\dim(C)+\dim(F)$, let $\dcal_{n}^{X}(C;F)$ be the differential operator $\dcal_{n}^{L(C)}(C;F)$ 
regarded as an operator on $X$ through the inclusion $\iota_{C}:L(C) \hookrightarrow X$, where the operator $\dcal_{n}^{L(C)}(C;F)$ is defined as in {\rm (1)} 
replacing $(X,\Lambda)$ by $(L(C),L(C) \cap \Lambda)$. 
\end{defin}

For unimodular cones $C$ in $X$ with $\dim(C) < \dim(X)$, the differential operator $\dcal_{n}^{X}(C;F)$ 
is characterized by the identity $\iota_{C}^{*}\dcal_{n}^{X}(C;F)\varphi=\dcal_{n}^{L(C)}(C;F)\iota_{C}^{*}\varphi$ for $\varphi \in C^{\infty}(X)$. 
Thus, a direct computation using Definition \ref{opcone}, \eqref{intbyp1}, \eqref{diffop1} 
combined with Proposition \ref{uC} (replacing $(X,\Lambda)$ by $(L(C),L(C) \cap \Lambda)$ if necessary) shows the following. 
\begin{theorem}
\label{coneA}
Let $C$ be a unimodular cone in the rational space $(X,\Lambda)$ with a 
rational inner product. Then, for any $\varphi \in \scal(X)$, we have 
\[
R_{N}(C;\varphi) \sim \sum_{n \geq 0}N^{-n}
\sum_{F \in \fcal(C),\,\dim (F) \geq \dim(C)-n}
\int_{F}\dcal_{n}^{X}(C;F)\varphi.  
\]
\end{theorem}

Let $C$ be a unimodular cone in the rational space $(X,\Lambda)$, and let $F \in \fcal(C)$. 
The order of the differential operator $\dcal_{n}^{X}(C;F)$ is $n-(\dim(C)-\dim(F))$, and which is equal to the order of 
the differential operator $\dcal_{n}^{X/L(F)}(\pi_{F}(C);0)$. Moreover, we have the following. 

\begin{lem}
\label{TransC1}
Let $C$ be a unimodular cone in $(X,\Lambda)$, and let $F \in \fcal(C)$. 
Then, the operator $\dcal_{n}^{X}(C;F)$ coincides with the lift of the operator $\dcal_{n}^{X/L(F)}(\pi_{F}(C);0)$ on $X$ through 
the identification $X/L(F) \cong L(F)^{\perp_{Q}}$. 
\end{lem}

\begin{proof}
Let $K$ be the integral basis of $L(C) \cap \Lambda$ generating $C$, and let $F=C(I_{F})$ with a subset $I_{F}$ of $K$. 
Then, the cone $\pi_{F}(C)$ in the rational space $(X/L(F),\pi_{F}(\Lambda))$ is a unimodular cone with 
the generator $\ol{I_{F}}=\{\ol{e}\,;\,\ol{e}=\pi_{F}(e),\ e \in I_{F}\}$. Thus, by Definition \ref{opcone}, we have 
\[
\dcal_{n}^{X/L(F)}(\pi_{F}(C);0)=(-1)^{n-\dim(\pi_{F}(C))}\sum_{\ol{I} \subset \ol{I_{F}}}
\sum_{\ol{\nu} \in \mb{Z}_{>0}^{\ol{I}}\,;\,|\ol{\nu}|=n}
p_{\ol{I}}(\ol{\nu})L(\ol{I_{F}};\ol{I},\ol{I_{F}};\ol{\nu}-e(\ol{I})). 
\]
The subsets $\ol{I}$ of $\ol{K}$ correspond to the subsets $I$ of $I_{F}$ by the projection $\pi_{F}$, and 
the elements $\ol{\nu}$ in $\mb{Z}_{+}^{\ol{I}}$ corresponds to the elements $\nu$ in $\mb{Z}_{+}^{I}$. 
Therefore, Lemma \ref{UD} shows 
\[
L(\ol{I_{F}};\ol{I},\ol{I_{F}};\ol{\nu}-e(\ol{I}))=L(\pi_{F}(I_{F});\pi_{F}(I),\pi_{F}(I_{F});\pi_{F}(\nu-e(I)))=L(K;I,I_{F};\nu-e(I))
\]
as an operator on $X$. From this, the assertion follows. 
\end{proof}

\begin{example}
In one dimension, it is easy to compute the differential operators $\dcal_{n}^{X}(C;F)$. 
Let $X$ be a $1$-dimensional vector space with the lattice $\Lambda$. 
Let $u \in \Lambda$ be a generator and set $C=\mb{R}_{+}u$. The faces of $C$ are $0$ and $C$ itself. 
Then, $E=\{u\}$. 
By definition, we have $\dcal_{0}^{X}(C;C)=1$, $\dcal_{n}^{X}(C;C)=0$ $(n \geq 1)$. 
By \eqref{opdef2}, we have $L(E;\{u\},\{u\};k)=\nabla_{u}^{k}$ for $k \in \mb{Z}_{+}$. 
Thus, by Definition \ref{opcone}, we have 
\begin{equation}
\label{1DconeN}
\dcal_{n}^{X}(C;0)=(-1)^{n-1}p(n)\nabla_{u}^{n-1}=-\frac{b_{n}}{n!}\nabla_{u}^{n-1},\quad n \geq 1, 
\end{equation}
and its symbol is given by $-\frac{b_{n}}{n!}\ispa{\xi,u}^{n-1}$, and 
we have $\displaystyle R_{N}(C;\varphi) \sim \int_{C}\varphi -\sum_{n \geq 1}\frac{b_{n}}{n!}\nabla_{u}^{n-1}\varphi(0)$.
\end{example}

\begin{theorem}
\label{BVvsT}
For each unimodular cone $C$ in a rational space $(X,\Lambda)$, each face $F \in \fcal(C)$ and 
each non-negative integer $n$ such that $\dim(F) \geq \dim(C)-n$, we have 
\[
\dcal_{n}^{X}(C;F)=D_{n}^{X}(C;F), 
\]
where $\dcal_{n}^{X}(C;F)$ is the differential operator defined in {\rm Definition \ref{opcone}} and 
$D_{n}^{X}(C;F)$ is the Berline-Vergne operator defined in {\rm Definition \ref{opconeG}}. 
\end{theorem}

\begin{proof}
For any rational space $(X,\Lambda)$, any rational subspace $L$ in $X$, any unimodular cone $C$ in $X/L$ and 
any non-negative integer $n$ satisfying $n \geq \dim(C)$, define the operator $\dcal_{n}^{X}(C)$ on $X$ by 
the lift of $\dcal_{n}^{X/L}(C;0)$ to $X$ under the identification $X/L \cong L^{\perp_{Q}}$. 
We need to check that these operators satisfy the conditions in Theorem \ref{BerVer}. 
The condition (1) in Theorem \ref{BerVer} follows from this definition. 
The condition (4) in Theorem \ref{BerVer} follows from Theorem \ref{coneA} and Lemma \ref{TransC1}. 
The condition (3) follows from {\it Example} above. The condition (2) follows from Definition \ref{opcone}. 
Therefore, the assertion follows from Theorem \ref{BerVer}. 
\end{proof}

\section{Asymptotic Euler-Maclaurin formula over rational cones}
\label{DRCONE}

In this section, we derive an asymptotic Euler-Maclaurin formula of $R_{N}(C;\varphi)$ for general rational cone $C$. 
To discuss asymptotic expansion of $R_{N}(C;\varphi)$ for pointed rational cones $C$, 
we define, for such a cone $C$ and non-negative integer $n$, the distribution $A_{n}(C;\cdot) \in \scal'(X)$ by 
\begin{equation}
\label{distA}
A_{n}(C;\varphi):=\sum_{F \in \fcal(C),\dim(F) \geq \dim(C)-n}\int_{F}D_{n}^{X}(C;F)\varphi,  
\end{equation}
where $D_{n}^{X}(C;F)$ is the Berline-Vergne operator defined in Definition \ref{opconeG}. 

\begin{lem}
\label{VdistA}
Let $\{C_{i}\}_{i=1}^{d}$ be a family of pointed rational cones in a rational space $(X,\Lambda)$ satisfying $\sum_{i}r_{i}\chi(C_{i})=0$, 
where, for each subset $S \subset X$, $\chi(S)$ denotes the characteristic function of $S$. 
We set $m=\max_{i}\dim(C_{i})$. 
Suppse further that there exists a vector $\eta \in X^{*}$ such that $\ispa{\eta,x}<0$ for each $\displaystyle 0 \neq x \in \cup_{i}C_{i}$. 
Then, for each $\varphi \in \scal(X)$, we have 
\begin{equation}
\label{VdistAeq}
\sum_{i,\dim(C_{i}) \geq m-n}r_{i}A_{n-m+\dim(C_{i})}(C_{i};\varphi)=0. 
\end{equation} 
\end{lem}

\begin{proof}
Since $A_{k}(C_{i};\cdot)$ are distributions and $C_{0}^{\infty}(X)$ is dense in $\scal(X)$, 
it is enough to prove \eqref{VdistAeq} for each $\varphi \in C_{0}^{\infty}(X)$. 
Note that the function $S(C)$ defined in \eqref{SandI} have a valuation property (see \cite{BP}). 
By this and the equation \eqref{ind1}, we have 
\[
\sum_{i}\sum_{G \in \fcal(C_{i})}r_{i}\mu (\pi_{G}(C_{i}))I(G)=0, 
\]
where the subscript $X/L(G)$ in $\mu_{X/L(G)}$ is dropped since these functions are lift to $X^{*}$. 
Substituting $t\xi$ ($t \in \mb{R}$, $\xi \in X^{*}$) in these functions and taking the Taylor expansion of each function, we have 
\[
\sum_{k \geq -m}\sum_{i,G \in \fcal(C_{i}),\dim(G)+k \geq 0}t^{k}\mu^{\dim(G)+k}(\pi_{G}(C_{i}))(\xi)I(G)(\xi)=0.  
\]
Thus, each coefficient of $t^{k}$ in the above vanishes, and hence we have 
\begin{equation}
\label{isym}
\sum_{i,G \in \fcal(C_{i}),\dim(G) \geq m-n}r_{i}\mu^{n-m+\dim(G)}(\pi_{G}(C_{i}))(i\xi +\eta)I(G)(i\xi +\eta)=0 
\end{equation}
for each $n \geq 0$, where, $\eta \in X^{*}$ is as in the statement of the lemma and $\xi \in X^{*}$ is arbitrary. 
Let $\varphi \in C_{0}^{\infty}(X)$. We have 
\begin{equation}
\label{chcont}
D_{n-m+\dim(C_{i})}(C_{i};G)\varphi (x)
=\frac{1}{(2\pi)^{m}}\int_{X^{*}}e_{i\xi+\eta}(x)\mu^{n-m+\dim(G)}(\pi_{G}(C_{i}))(i\xi+\eta)\hat{\varphi}(\xi-i\eta)\,d\xi,
\end{equation}
where the Lebesgue measure $d\xi$ on $X^{*}$ is normalized as in Subsection \ref{Heu}. 
Taking the integral over $G$, we have 
\begin{equation}
\label{dverI}
\int_{G}D_{n-m+\dim(C_{i})}(C_{i};G)\varphi=(2\pi)^{-m}\int_{X^{*}}I(G)(i\xi+\eta)\mu^{n-m+\dim(G)}(\pi_{G}(C_{i}))(i\xi+\eta)\hat{\varphi}(\xi-i\eta)\,d\xi, 
\end{equation}
where we have used the fact that $e_{i\xi+\eta}$ is integrable on $G$ for each $i$ and $G \in \fcal(C_{i})$. 
Thus, multiplying \eqref{dverI} by $r_{i}$, taking the sum over all $i$ and $G \in \fcal(C_{i})$ with $\dim(G) \geq m-n$ and 
using the equation \eqref{isym}, we have \eqref{VdistAeq}. 
\end{proof}

\begin{theorem}
\label{coneAG}
Let $C$ be a pointed rational cone in a rational space $(X,\Lambda)$ with a rational inner product $Q$. 
Then, for any $\varphi \in \scal(X)$, we have 
\[
R_{N}(C;\varphi) \sim \sum_{n \geq 0}N^{-n}A_{n}(C;\varphi),
\]
where $A_{n}(C;\varphi)$ is defined in $\eqref{distA}$. 
Furthermore, the uniqueness statement of {\rm Theorem \ref{BerVer}} still true if we replace 
the unimodular cones in the statement of {\rm Theorem \ref{BerVer}} with the pointed rational cones. 
\end{theorem}

\begin{proof}
By replacing $X$ with $L(C)$, we may assume that $m:=\dim(C)=\dim(X)$. 
It is well-known that, for any pointed rational cone $C$, one can find a finite set of unimodular cones $\ccal=\{\sigma_{i}\}_{i=1}^{d}$ 
such that $\ccal$ is a subdivision of the pointed cone $C$, namely, the collection $\ccal$ satisfies the following. 
\[
(1) \hspace{5pt} \ccal=\bigcup_{\sigma \in \ccal}\fcal(\sigma) \qquad
(2) \hspace{5pt} \sigma,\tau \in \ccal \Longrightarrow \sigma \cap \tau \in \fcal(\sigma) \cap \fcal(\tau)  \qquad 
(3) \hspace{5pt} C=\bigcup_{\sigma \in \ccal}\sigma 
\]
(For a proof of this fact, see \cite{F}, Section 2.6.) 
By the inclusion-exclusion principle, there is a relation $\chi(C)=\sum_{\sigma \in \ccal}r_{\sigma}\chi(\sigma)$ with some $r_{\sigma} \in \mb{Z}$. 
Then, we have 
\[
R_{N}(C;\varphi)
=\frac{1}{N^{m}}\sum_{\sigma \in \ccal}\sum_{\gamma \in \Lambda}r_{\sigma}\chi(\sigma)(\gamma)\varphi(\gamma/N) 
=\sum_{\sigma \in \ccal}N^{-m+\dim(\sigma)}r_{\sigma}R_{N}(\sigma;\varphi). 
\]
Each cone $\sigma \in \ccal$ is a unimodular cone in $X$, and hence 
we can apply Theorem \ref{coneA} to $R_{N}(\sigma;\varphi)$ for each $\sigma \in \ccal$. 
By a direct computation, we have 
\[
R_{N}(C;\varphi) \sim \sum_{n \geq 0}N^{-n}\sum_{\sigma \in \ccal,\dim(\sigma) \geq m-n}r_{\sigma}A_{n-m+\dim(\sigma)}(\sigma;\varphi), 
\]
and hence Lemma \ref{VdistA} shows the first part of the assertion. 
The last assertion on the uniqueness follows from the same discussion as in Theorem \ref{BerVer}, and 
hence we omit the proof. 
\end{proof}

In the next section, we need the following lemma, which generalizes Lemma \ref{notuC}. 

\begin{lem}
\label{notpt}
Let $C$ be a rational cone in a rational space $(X,\Lambda)$. Let $L=C \cap (-C)$. 
Then, for any $\varphi \in \scal(X)$, we have 
\[
R_{N}(C;\varphi)=R_{N}(\pi_{L}(C);(\pi_{L})_{*}\varphi)+O(N^{-\infty}). 
\]
\end{lem}

\begin{proof}
If $L=\{0\}$, we have the conclusion without the term $O(N^{-\infty})$. 
So, we assume that $L \neq \{0\}$. For simplicity, we write $\pi=\pi_{L}:X \to X/L$, the natural projection. 
Take $\varphi \in C_{0}^{\infty}(X)$. 
Since $L$ is rational, one can take a complementary rational subspace $W$ to $L$ such that $X=L \oplus W$ and $\Lambda=(L \cap \Lambda) \oplus (W \cap \Lambda)$. 
Set $G=C \cap W$, which is a pointed rational cone in $W$. We have $C=L +G$. 
Take a subdivision $\ccal$ of $G$ into unimodular cone in $W$. The set $\{C_{\sigma}=L +\sigma\,;\,\sigma \in \ccal\}$ 
is a subdivision of $C$ into rational cones. 
Then, there is a relation $\chi(C)=\sum_{\sigma \in \ccal}r_{\sigma}\chi(C_{\sigma})$, and hence 
\[
R_{N}(C;\varphi)=\sum_{\sigma \in \ccal}N^{-m+\dim(C_{\sigma})}r_{\sigma}R_{N}(C_{\sigma};\varphi). 
\]
Note that the cones $C_{\sigma}$ is of the form discussed in Lemma \ref{notuC}. 
By Lemma \ref{notuC}, we have $R_{N}(C_{\sigma};\varphi)=R_{N}(\pi(\sigma);\pi_{*}\varphi)+O(N^{-\infty})$, and hence 
\[
R_{N}(C;\varphi)=\sum_{\sigma \in \ccal}N^{-m+\dim(\sigma)+\dim(L)}r_{\sigma}R_{N}(\pi(\sigma);\pi_{*}\varphi)+O(N^{-\infty}). 
\]
The set $\{\pi(\sigma)\,;\,\sigma \in \ccal\}$ is a subdivision of the pointed rational cone $\pi(C)$ in $X/L$ into unimodular cones. 
Furthermore, since $\chi(C)=\sum_{\sigma \in \ccal}r_{\sigma}\chi(C_{\sigma})$ we have $\chi(\pi(C))=\sum_{\sigma \in \ccal}r_{\sigma}\chi(\pi(C_{\sigma}))$. 
($\pi$ defines a valuation. See \cite{BP}.) 
Thus, the sum in the right hand side of the last equation is $R_{N}(\pi(C);\pi_{*}\varphi)$, which proves the assertion. 
\end{proof}

\section{Results and their proofs}
\label{RPC}

In this section, we restate Theorem \ref{lAEMT} on the asymptotic Euler-Maclaurin formula of the Riemann sum 
\[
R_{N}(P;\varphi):=\frac{1}{N^{\dim(P)}}\sum_{\gamma \in (NP) \cap \Lambda}\varphi(\gamma/N), 
\]
for a lattice polytope $P$ in a rational space $(X,\Lambda)$ and a smooth function $\varphi$ on $P$ 
in the abstract notation we used as before 
and give its proof. We also state and give proofs of its corollaries.

\subsection{Main theorems and their proofs}
\label{MTP}

\begin{theorem}
\label{AEMTh}
Let $P$ be a lattice polytope in a rational space $(X,\Lambda)$ with a rational inner product. 
For each $f \in \fcal(P)$ and $n \in \mb{Z}_{+}$ satisfying $\dim(f) \geq \dim(P)-n$, let $D_{n}^{X}(P;f)$ 
be the differential operator defined in {\rm Definition \ref{opconeG}}. 
Then, for each $\varphi \in C^{\infty}(P)$, we have the following asymptotic expansion$:$
\begin{equation}
\label{AEMc}
R_{N}(P;\varphi) \sim \sum_{n \geq 0}A_{n}(P;\varphi)N^{-n}, \qquad 
A_{n}(P;\varphi) =\sum_{f \in \fcal(P)\,;\,\dim(f) \geq \dim(P)-n}\int_{f}D_{n}^{X}(P;f)\varphi. 
\end{equation}
\end{theorem}

To prove Theorem \ref{AEMTh}, we need the following lemma. 

\begin{lem}
\label{KEY}
Let $f \in \fcal(P)$ and let $n \in \mb{Z}_{+}$ satisfy $\dim(f) \geq \dim(P)-n$. Then, for any $g \in \fcal(P)$ such that $g \subset f$, we have 
\begin{equation}
\label{key10}
D_{n}^{X}(P;f)=D_{n}^{X}(\pi_{g}(C_{P}(g));\pi_{g}(C_{f}(g))). 
\end{equation}
\end{lem}

\begin{proof}
First of all, as in Subsection \ref{BVOP}, note that we have $D_{n}^{X}(P;f)=D_{n}^{X}(\pi_{f}(C_{P}(f));0)$. 
We set $C=\pi_{g}(C_{P}(g))$ and $G=\pi_{g}(C_{f}(g))$. Then, $C$ is a pointed rational cone in $X/L(g)$ and $G \in \fcal(C)$. 
Furthermore, we have $D_{n}^{X}(\pi_{g}(C_{P}(g));\pi_{g}(C_{f}(g)))=D_{n}^{X}(\pi_{G}(C);0)$, 
where $\pi_{G}:X/L(g) \to (X/L(g))/L(G)=X/L(f)$ is the natural projection. 
Since $\pi_{G} \comp \pi_{g}=\pi_{f}:X \to X/L(f)$ and $C_{P}(f)=L(f)+C_{P}(g)$, 
we have $\pi_{G}(C)=\pi_{f}(C_{P}(g))=\pi_{f}(C_{P}(f))$, and hence the equation \eqref{key10} follows. 
\end{proof}

\noindent{{\it Proof of} Theorem \ref{AEMTh}.}\hspace{5pt}
For any $g \in \fcal(P)$ and $v \in g$, we set $C_{P}^{+}(g)=C_{P}(g)+v$ which 
does not depend on the choice of $v \in g$. 
Then, we use the following version of Euler's formula (\cite{BrV}, Proposition 3.2, (1)): 
\begin{equation}
\label{Eu}
\delta((NP) \cap \Lambda)=\sum_{g \in \fcal(P)}(-1)^{\dim(g)}\delta(C_{NP}^{+}(Ng) \cap \Lambda), 
\end{equation}
where, $N$ is a positive integer and, for any subset $S$ of $\Lambda$, $\delta(S)$ is a distribution defined by 
\[
\ispa{\delta(S),\varphi}=\sum_{s \in S}\varphi(s),\quad \varphi \in C_{0}^{\infty}(X). 
\]
For each $N \in \mb{Z}_{>0}$ and $\varphi \in C^{\infty}(X)$, 
we set $(D_{1/N}^{*}\varphi)(x)=\varphi(x/N)$. For each $g \in \fcal(P)$, we fix $v_{g} \in g \cap \Lambda$. 
Clearly we have 
\[
\begin{gathered}
\ispa{\delta(NP \cap \Lambda),D_{1/N}^{*}\varphi}=N^{\dim(P)}R_{N}(P;\varphi), \\
\ispa{\delta(C_{NP}^{+}(Nf) \cap \Lambda),\,D_{1/N}^{*}\varphi}
=N^{\dim(P)}R_{N}(C_{P}(g);T_{v_{g}}\varphi), 
\end{gathered}
\]
where, for $v \in X$, we set $T_{v}\varphi(x)=\varphi(v+x)$. 
Take $\varphi \in C^{\infty}(P)$ and extend $\varphi$ as a compactly supported smooth function on $X$. 
Then, by \eqref{Eu} and Lemma \ref{notpt}, we have 
\begin{equation}
\label{aux51}
\begin{split}
R_{N}(P;\varphi) &=\sum_{g \in \fcal(P)}(-1)^{\dim(g)}R_{N}(C_{P}(g);T_{v_{g}}\varphi) \\
& \sim \sum_{g \in \fcal(P)}(-1)^{\dim(g)}R_{N}(\pi_{g}(C_{P}(g));(\pi_{g})_{*}T_{v_{g}}\varphi). 
\end{split}
\end{equation}
Since $\pi_{g}(C_{P}(g))$ is a pointed rational cone in $X/L(g)$ with respect to the lattice $\pi_{g}(\Lambda)$, 
we can use Theorem \ref{coneAG} for $R_{N}(\pi_{g}(C_{P}(g));(\pi_{g})_{*}T_{v_{g}}\varphi))$ and hence 
\begin{equation}
\label{coeff}
\begin{gathered}
R_{N}(P;\varphi) \sim 
\sum_{n \geq 0}A_{n}(P;\varphi)N^{-n},\\
A_{n}(P;\varphi)=\sum_{g \in \fcal(P)}
\sum_{
\stackrel{G \in \fcal(\pi_{g}(C_{P}(g))),}{\dim(G) \geq \dim(P)-n-\dim(g)}
}
(-1)^{\dim(g)}\int_{G}D_{n}^{X}(\pi_{g}(C_{P}(g));G)(\pi_{g})_{*}T_{v_{g}}\varphi. 
\end{gathered}
\end{equation}
Each faces $G \in \fcal(\pi_{g}(C_{P}(g)))$ with $\dim(G) \geq \dim(P)-n-\dim(g)$ 
can be written as $G=\pi_{g}(C_{f}(g))$ with a face $f \in \fcal(P)$ 
such that $g \subset f$ and $\dim(f) \geq \dim(P)-n$. 
Furthermore, the correspondence 
\[
\{f \in \fcal(P)\,;\,g \subset f\} \ni f \mapsto \pi_{g}(C_{f}(g)) \in \fcal(\pi_{g}(C_{P}(g)))
\]
defines a bijective correspondence between the above two sets. 
Thus, by Lemma \ref{KEY} and the definition of 
the function $(\pi_{g})_{*}T_{v_{f}}\varphi$, we can write 
\[
\begin{split}
A_{n}(P;\varphi)&=
\sum_{g \in \fcal(P)}\sum_{
\stackrel{f \in \fcal(P),\,g \subset f}{\dim(f) \geq \dim(P)-n}
}
(-1)^{\dim(g)}\int_{\pi_{g}(C_{f}(g))}D_{n}^{X}(P;f)(\pi_{g})_{*}T_{v_{g}}\varphi \\
&=\sum_{g \in \fcal(P)}\sum_{
\stackrel{f \in \fcal(P),\,g \subset f}{\dim(f) \geq \dim(P)-n}
}
(-1)^{\dim(g)}\int_{C_{f}^{+}(g)}D_{n}^{X}(P;f)\varphi \\
&=\sum_{f \in \fcal(P),\,\dim(f) \geq \dim(P)-n}\sum_{g \in \fcal(f)}
(-1)^{\dim(g)}\int_{\ispa{f}}\chi^{f}(C_{f}^{+}(g))D_{n}^{X}(P;f)\varphi,  
\end{split}
\]
where $\ispa{f}$ is the affine hull of $f$, and for each $S \subset \ispa{f}$, we denote $\chi^{f}(S)$ the 
characteristic function of $S$ on $\ispa{f}$. 
In the first line above, we used an obvious identity $D_{n}^{X}(P;f)(\pi_{g})_{*}\psi=(\pi_{g})_{*}D_{n}^{X}(P,f)\psi$ 
for $\psi \in C_{0}^{\infty}(X)$. To simplify the above, we use the formula (Proposition 3.1, (1) in \cite{BrV})
\begin{equation}
\label{decoC}
\sum_{g \in \fcal(P)}(-1)^{\dim(g)}\chi(C_{P}^{+}(g))=\chi(P).  
\end{equation}
Note that in \cite{BrV}, the above formula 
is proved for $P$ with non-empty interior. Replacing $P$ by $f \in \fcal(P)$, 
which is regarded as a polytope in the affince subspace $\ispa{f}$ with non-empty relative interior, we have 
\[
\sum_{g \in \fcal(f)}(-1)^{\dim(g)}\chi^{f}(C_{f}^{+}(g))=\chi^{f}(f). 
\]
Therefore, we obtain the formula \eqref{AEMc} for $A_{n}(P;\varphi)$, 
which complete the proof of Theorem \ref{AEMTh}. 
\hfill$\square$

\vspace{10pt}

\begin{rem}
In the proof above, we used Theorem \ref{coneAG}. However, if the lattice polytope $P$ is Delzant, 
then the cone $\pi_{f}(C_{P}(f))$ for each $f \in \fcal(P)$ is a unimodular cone in $X/L(f)$. 
Therefore, we only need to use Theorem \ref{coneA}. Hence, for Delzant lattice polytopes, 
it turns out that our proof of Theorem \ref{AEMTh} is independent of \cite{BeV}. 
However, for general lattice polytopes, it does not seem to be easy to 
construct the operator $D_{n}^{X}(P;f)$ in such a way given in Definition \ref{opcone}. 
Indeed, Definition \ref{opcone} is based on Proposition \ref{uC}. This means that 
if we could obtain a result like Proposition \ref{uC} for general rational cones, then 
one might be able to find such an expression as in Definition \ref{opcone}. Hence, it might be 
better to prove a result like Proposition \ref{uC} for rational cones without using a subdivision of 
rational cones into unimodular cones. However, to do this, it seems that one need to find a different method. 
\end{rem}

Next, we show that, under some assumptions, the asymptotic expansion 
of $R_{N}(P;\varphi)$ of the form $\eqref{AEMc}$ is unique. 

\begin{theorem}
\label{UofP}
Suppose that, for any rational space $(X,\Lambda)$ with a rational inner product, 
rational subspace $L$ of $X$, pointed rational cone $C$ in $X/L$ 
and non-negative integer $n$ such that $n \geq \dim(C)$, there exists a homogeneous differential operator $\dcal_{n}^{X}(C)$ 
of order $n-\dim(C)$ with symbol $\nu_{n}^{X}(C)$ such that they satisfy the conditions $(1)$, $(2)$ and $(3)$ in {\rm Theorem \ref{BerVer}}. 
Furthermore, suppose that, 
for any lattice polytope $P$ in $X$ and $\varphi \in C^{\infty}(P)$, the following holds$:$
\begin{equation}
\label{AAP}
R_{N}(P;\varphi) \sim \sum_{n \geq 0}N^{-n}\sum_{f \in \fcal(P)\,;\,\dim(f) \geq \dim(P) -n}
\int_{f}\dcal_{n}^{X}(\pi_{f}(C_{P}(f)))\varphi.   
\end{equation}
Then, we have $\dcal_{n}^{X}(C)=D_{n}^{X}(C;0)$ for any pointed rational cone $C$ in $X$ and 
non-negative integer $n$ with $n \geq \dim(C)$, 
where the operator $D_{n}^{X}(C;0)$ is defined in {\rm Definition \ref{opconeG}}. 
\end{theorem}

\begin{proof}
Let us prove the assertion by the induction on $\dim(X)$. For $\dim(X)=0,1$, the assertion is true by the condition (3) in Theorem \ref{BerVer}. 
Suppose that for each $(X,\Lambda)$ with $\dim(X) \leq m-1$, the assertion holds. 
Let $\dim(X)=m$. Take a pointed rational cone $C$ in $X$. We may assume that $\dim(C)=m$. 
Take a vector $\xi \in \Lambda^{*}$ such that $\ispa{\xi,x}>0$ for any $x \in C$. 
Set $P_{1}=C \cap \{x \,;\, \ispa{\xi,x} \leq 1\}$, which is a rational polytope in $X$. 
Hence, each vertex of $P_{1}$ is a rational point in $X$. 
We take a positive integer $q$ such that $P=qP_{1}$ is a lattice polytope. 
Let $U$ be a small open ball around the origin such that $U \cap \vcal(P)=\{0\}$ and $U \subset \{x\,;\,\ispa{\xi,x}<q\}$. 
Then, by the assumption, for each $\varphi \in C_{0}^{\infty}(U)$, 
the Riemann sum $R_{N}(P;\varphi)$ admits the asymptotic expansion \eqref{AAP}. 
In \eqref{AAP}, if $\dim(f)>0$, then since $\dim(\pi_{f}(C_{P}(f)))=m-\dim(f)<m$, the differential operators $\dcal_{n}^{X}(\pi_{f}(C_{P}(f)))$ 
coincide with $D_{n}^{X}(P,f)=D_{n}^{X}(\pi_{f}(C_{P}(f)),0)$ by the induction hypothesis. 
Take a vertex $v$ of $P$. Suppose $v \neq 0$. Since $\varphi$ is zero near $v$, the contribution from the vertex $v$ to the 
expansion \eqref{AAP} vanishes. Thus, by Theorem \ref{AEMTh}, we have 
\begin{equation}
\label{unop}
[\dcal_{n}^{X}(\pi_{0}(C_{P}(0)))\varphi] (0)=[D_{n}^{X}(P;0)\varphi](0)\quad (n \geq m). 
\end{equation}
Take $\rho \in C_{0}^{\infty}(U)$ such that $\rho=1$ near $0$. For any $\psi \in C^{\infty}(X)$, 
we have \eqref{unop} for $\varphi=\rho \psi$. But since $\rho=1$ near $0$, the equation \eqref{unop} holds for any $\varphi \in C^{\infty}(X)$. 
Take $\varphi \in C^{\infty}(X)$ and $x \in X$. Applying \eqref{unop} for the function $T_{x}\varphi$, we have 
\[
[\dcal_{n}^{X}(\pi_{0}(C_{P}(0)))\varphi] (x)=[D_{n}^{X}(P;0)\varphi](x)\quad (n \geq m) 
\]
for any $\varphi \in C^{\infty}(X)$. Since $\pi_{0}(C_{P}(0))=C$, we conclude the assertion. 
\end{proof}

\subsection{Computation in one and two dimensions}
\label{Comp2D}

A polytope $P$ in a rational space $(X,\Lambda)$ is said to be Delzant if for each vertex $v$ of $P$, the number 
of edges incident to $v$ is $\dim(X)$ and there exists an integral basis $E$ of $\Lambda$ such that 
each edge incident to $v$ is of the form $\{v +t e\,;\,t \geq 0\}$ with an $e \in E$. 
In this and the next subsection, we give explicit computations for Delzant lattice polytopes. 
To compute each coefficient $A_{n}(P;\varphi)$ in the asymptotic expansion of the Riemann sum $R_{N}(P;\varphi)$, 
it is important to compute in low dimensions. In this subsection, we perform these computation. 
In this and the next subsections, we drop the superscript $X$ in $D_{n}^{X}(P,g)$.

\subsubsection{In one dimension}
Let $X$ be a $1$-dimensional vector space 
with the lattice $\Lambda$. Let $u \in \Lambda$ be a generator and set $C=\mb{R}_{+}u$. 
We have computed the differential operator $D_{n}(C;0)$ in {\it Example} at the end of Subsection \ref{DOC}. 
Let $P$ be an interval given by $P=\{tu \in X\,;\,a \leq t \leq b\}$ with $a,b \in \mb{Z}$, $a<b$.  
Since $D_{n}(P;P)=0$ for $n \geq 1$, we have 
\[
A_{n}(P;\varphi)=D_{n}(P;\{a\})\varphi (a) +D_{n}(P;\{b\})\varphi (b)
=(-1)^{n-1}p(n)[\nabla_{u}^{n-1}\varphi (a) +\nabla_{-u}^{n-1}\varphi (b)]
\]
Identifying $X=\mb{R}$ and $u=1$ so that $\Lambda =\mb{Z}$, we have 
\[
A_{n}(P;\varphi)=-\frac{b_{n}}{n!}[\varphi^{(n-1)}(a)-(-1)^{n}\varphi^{(n-1)}(b)].   
\]
Substituting $b_{2m+1}=0$ $(m \geq 1)$ and $b_{2m}=(-1)^{m-1}B_{m}$ with the Bernoulli number $B_{m}$, 
we have 
\[
A_{2m+1}(P;\varphi)=0,\quad 
A_{2m}(P;\varphi)=(-1)^{m-1}\frac{B_{m}}{(2m)!}[\varphi^{(2m-1)}(b)-\varphi^{(2m-1)}(a)], 
\]
which shows the classical asymptotic Euler-Maclaurin formula. 

\subsubsection{In two dimension}
Next, we compute in two dimension. Let $(X,\Lambda)$ be a two dimensional rational vector space 
with a rational inner product $Q$. 
Let $E=\{e_{1},e_{2}\}$ be an integral basis of the lattice $\Lambda$, and set $C=\mb{R}_{+}e_{1}+\mb{R}_{+}e_{2}$.
Set
\[
e_{1}=u_{1}+c_{1}e_{2},\quad e_{2}=u_{2}+c_{2}e_{1}, 
\]
where the non-zero vectors $u_{1},u_{2} \in X$ satisfy $Q(u_{1},e_{2})=Q(u_{2},e_{1})=0$, and $c_{1}, c_{2} \in \mb{Q}$ are 
given by 
\begin{equation}
\label{ccc}
c_{1}=\frac{Q(e_{1},e_{2})}{Q(e_{2},e_{2})},\quad c_{2}=\frac{Q(e_{1},e_{2})}{Q(e_{1},e_{1})}. 
\end{equation}
Define $\lambda_{1},\lambda_{2} \in \mb{Z}^{E}$ by $\lambda_{i}(e_{j})=\delta_{i,j}$. 
A straightforward computation shows 
\[
\begin{gathered}
L(C;E,E;k\lambda_{1}+l\lambda_{2})=\nabla_{e_{1}}^{k}\nabla_{e_{2}}^{l},\\
L(C;\{e_{1}\},\{e_{1}\};k\lambda_{1})=\nabla_{u_{1}}^{k},\quad
L(C;\{e_{2}\},\{e_{2}\};l\lambda_{2})=\nabla_{u_{2}}^{l},\\
L(C;\{e_{1}\},E;k\lambda_{1})=c_{1}\sum_{s=0}^{k-1}\nabla_{u_{1}}^{s}\nabla_{e_{1}}^{k-1-s},\quad
L(C;\{e_{2}\},E;l\lambda_{2})=c_{2}\sum_{s=0}^{l-1}\nabla_{u_{2}}^{s}\nabla_{e_{2}}^{l-1-s}. 
\end{gathered}
\]
Set $F_{1}=\mb{R}_{+}e_{2}$, $F_{2}=\mb{R}_{+}e_{1}$. Then, we have 
\[
\begin{split}
D_{n}(C;F_{1})=&(-1)^{n-1}p(n)\nabla_{u_{1}}^{n-1},\quad 
D_{n}(C;F_{2})=(-1)^{n-1}p(n)\nabla_{u_{2}}^{n-1}\ \ (n \geq 1), \\
D_{n}(C;0)=&(-1)^{n}\sum_{k=1}^{n-1}p(k)p(n-k)\nabla_{e_{1}}^{k-1}\nabla_{e_{2}}^{n-1-k}\\
&+(-1)^{n}p(n)\left(
c_{1}\sum_{s=0}^{n-2}\nabla_{u_{1}}^{s}\nabla_{e_{1}}^{n-2-s}+c_{2}\sum_{s=0}^{n-2}\nabla_{u_{2}}^{s}\nabla_{e_{2}}^{n-2-s} 
\right) 
\quad (n \geq 2).
\end{split}
\]
Let $P$ be a Delzant lattice polytope in $(X,\Lambda)$. 
For each facet $f$ of $P$, $D_{n}(P;f)$ is the lift of $D_{n}(\pi_{f}(C_{P}(f));0)$. 
Let $\alpha_{f} \in \Lambda$ be the inward primitive normal of $f$. 
(Such a vector $\alpha_{f}$ exists because the dual basis of an integral basis of $\Lambda$ with 
respect to $Q$ is rational.) 
We identify $\pi_{f}(C_{P}(f))$ with $\mb{R}_{+}\alpha_{f}$ by the map
\[
\varphi_{f}:X/L(f) \ni x+L(f) \mapsto \frac{Q(x,\alpha_{f})}{Q(\alpha_{f},\alpha_{f})}\alpha_{f} \in \mb{R}\alpha_{f}. 
\]
Let $e_{1} \in \Lambda$ be a generator of $L(f) \cap \Lambda$. 
Since $P$ is Delzant, we can find $e_{2} \in C_{P}(f) \cap \Lambda$ such that $\{e_{1},e_{2}\}$ forms an integral basis of $\Lambda$. 
Then, the vector 
\begin{equation}
\label{genPP}
u_{f}:=\frac{Q(e_{2},\alpha_{f})}{Q(\alpha_{f},\alpha_{f})}\alpha_{f}
\end{equation}
is a generator of $\varphi_{f}(\pi_{f}(\Lambda))$ such that $\varphi_{f}(\pi_{f}(C_{P}(f)))=\mb{R}_{+}u_{f}$. 
Note that the definition of $u_{f}$ does not depend on the choice of $e_{2} \in C_{P}(f) \cap \Lambda$ 
whenever $e_{1},e_{2}$ forms an integral basis of $\Lambda$. 
Hence, by \eqref{1DconeN}, the differential operator $D_{n}(P;f)$ is given by 
\[
D_{n}(P;f)=(-1)^{n-1}p(n)\nabla_{u_{f}}^{n-1}=-\frac{b_{n}}{n!}\nabla_{u_{f}}^{n-1} \ \ (n \geq 1). 
\]
Therefore, we have the following. 
\begin{cor}
\label{2Dasympt}
Let $(X,\Lambda)$ be a two dimensional rational vector space with a rational inner product $Q$. 
Let $P$ be a Delzant lattice polytope in $(X,\Lambda)$. 
Then, the coefficients $A_{n}(P;\varphi)$ $(n \geq 2)$ in the asymptotic expansion $\eqref{AEMc}$ of the 
Riemann sum $R_{N}(P;\varphi)$ is given by 
\[
A_{n}(P;\varphi)=\sum_{f \in \fcal(P)_{1}}\int_{f}D_{n}(P;f)\varphi+
\sum_{v \in \vcal(P)}D_{n}(P;v)\varphi (v). 
\]
In the above, the differential operators $D_{n}(P;f)$ and $D_{n}(P;v)$ are given by 
\[
\begin{split}
D_{n}(P;f)&=-\frac{b_{n}}{n!}\nabla_{u_{f}}^{n-1},\\
D_{n}(P;v)&=\sum_{k=1}^{n-1}\frac{b_{k}b_{n-k}}{k!(n-k)!}\nabla_{e_{1}(v)}^{k-1}\nabla_{e_{2}(v)}^{n-1-k}\\
&+\frac{b_{n}}{n!}\left(
c_{1}(v)\sum_{s=0}^{n-2}\nabla_{u_{1}(v)}^{s}\nabla_{e_{1}(v)}^{n-2-s}+c_{2}(v)\sum_{s=0}^{n-2}\nabla_{u_{2}(v)}^{s}\nabla_{e_{2}(v)}^{n-2-s} 
\right), 
\end{split}
\]
where, for a face $f \in \fcal(P)_{1}$, $u_{f} \in X_{\mb{Q}}$ denotes the inward normal defined in \eqref{genPP}, 
and for a vertex $v \in \vcal(P)$, the vectors $e_{1}(v), e_{2}(v) \in \Lambda$ denote the integral basis of $\Lambda$ such that two facets meeting at $v$ 
lie on the half lines $v+te_{i}(v)$, $t \geq 0$, $i=1,2$, and $u_{1}(v),u_{2}(v) \in X$ satisfy
\[
\begin{gathered}
e_{1}(v)=u_{1}(v)+c_{1}(v)e_{2},\ \ Q(u_{1}(v),e_{2}(v))=0,\ \ c_{1}(v)=\frac{Q(e_{1}(v),e_{2}(v))}{Q(e_{2}(v),e_{2}(v))}, \\
e_{2}(v)=u_{2}(v)+c_{2}(v)e_{1},\ \ Q(u_{2}(v),e_{1}(v))=0,\ \ c_{2}(v)=\frac{Q(e_{1}(v),e_{2}(v))}{Q(e_{1}(v),e_{1}(v))}. 
\end{gathered}
\]
\end{cor}
Note that, in the following, we use $D_{2}(C;0)$ for two dimensional unimodular cone $C$. 
The explicit formula for $D_{2}(C;0)$ is given by 
\begin{equation}
\label{2Dcone2}
D_{2}(C;0)=p(1)^{2}+(c_{1}+c_{2})p(2)=\frac{1}{4}+(c_{1}+c_{2})\frac{1}{12},  
\end{equation}
where $c_{1},c_{2}$ are given in \eqref{ccc}. 

\subsection{Computation of the coefficient in the third term}
\label{Comp3}

Our main Theorem \ref{AEMTh}, or rather the construction of the operators $D_{n}(P;f)$, 
allows us to compute the coefficient $A_{2}(P;\varphi)$ in the third term of the asymptotic expansion 
\eqref{AEMc}. Before computing the third term, let us compute the first and second terms. 

\begin{cor}
\label{A0A1}
For any Delzant lattice polytope $P$ in a rational space $(X,\Lambda)$ 
with a rational inner product $Q$, we have 
\[
A_{0}(P;\varphi)=\int_{P}\varphi\,dx,\quad A_{1}(P;\varphi)=\frac{1}{2}\sum_{g \in \fcal(P)_{m-1}}\int_{g}\varphi, 
\]
where the integration on facets $g \in \fcal(P)_{m-1}$ is performed with respect to the measure on $g$ 
induced by the lattice $\Lambda$. 
\end{cor}

\begin{proof}
The first term is obvious. For the second term $A_{1}(P;\varphi)$, note that the dimension of faces which contribute to $A_{1}(P;\varphi)$ 
is $m-1$ and $m$. But the operator $D_{1}(P;P)$ is the lift of $D_{1}(0;0)$ (see Definition \ref{opcone}) which is zero. 
Thus, the contribution to $A_{1}(P;\varphi)$ comes from facets. Let $g \in \fcal(P)_{m-1}$. Then the operator $D_{1}(P;g)$ is the lift 
of $D_{1}(\pi_{g}(C_{P}(g));0)$, which is a rational constant. 
Let $\alpha_{g} \in \Lambda$ be inward primitive normal of the facet $g$. 
As in the computation in two dimension, let $\varphi_{g}:X/L(g) \to L(g)^{\perp_{Q}}$ be the isomorphism defined by 
\[
\varphi_{g}(x+L(g))=\frac{Q(x,\alpha_{g})}{Q(\alpha_{g},\alpha_{g})}\alpha_{g}. 
\]
We take an integral basis $e_{1},\ldots,e_{m-1}$ of $L(g) \cap \Lambda$. 
Since $P$ is Delzant, one can take $e_{m} \in C_{P}(g)$ such that $e_{1},\ldots,e_{m}$ form an integral basis of $\Lambda$. 
We set 
\begin{equation}
\label{genPPP}
u_{g}=\frac{Q(e_{m},\alpha_{g})}{Q(\alpha_{g},\alpha_{g})}\alpha_{g} \in L(g)^{\perp_{Q}}. 
\end{equation}
As before, the definition of $u_{g}$ above does not depend on the choice of $e_{m}$ above. 
By \eqref{1DconeN}, we have $D_{1}(\pi_{g}(C_{P}(g));0)=-\frac{b_{1}}{1!}=\frac{1}{2}$. 
Hence, we have
\[
A_{1}(P;\varphi)=\frac{1}{2}\sum_{g \in \fcal(P)_{m-1}}\int_{g}\varphi, 
\]
which completes the proof. 
\end{proof}

Note that the above formula for the second term $A_{1}(P;\varphi)$ coincides with that in \eqref{BB}. 
Indeed, if $X=\mb{R}^{m}$, $\Lambda=\mb{Z}^{m}$ and $Q$ is the standard Euclidean inner product, 
the primitive inward primitive normal $\alpha_{g}$ for each facet $g$ of a Delzant polytope $P$ 
is a part of an integral basis of $\mb{Z}^{m}$. 

Next, we compute the third term, which does not seem to have been obtained before. 
For simplicity, we work in the Euclidean space $X=\mb{R}^{m}$ with the standard lattice $\mb{Z}^{m}$ and 
the standard inner product. 

\begin{cor}
\label{thirdT}
Let $P$ be a Delzant lattice polytope in the Euclidean space $(\mb{R}^{m},\mb{Z}^{m})$ with 
the standard inner product $Q$.  
Then, we have the following: 
\[
\begin{split}
A_{2}(P;\varphi)=-&\frac{1}{12}\sum_{g \in \fcal(P)_{m-1}}\frac{1}{Q(\alpha_{g},\alpha_{g})}
\int_{g}\nabla_{\alpha_{g}}\varphi \\
+&
\sum_{g \in \fcal(P)_{m-2}}
\left[
\frac{1}{4}-\frac{1}{12}
\left(\frac{Q(\alpha_{1}(g),\alpha_{2}(g))}{Q(\alpha_{1}(g),\alpha_{1}(g))}
+\frac{Q(\alpha_{1}(g),\alpha_{2}(g))}{Q(\alpha_{2}(g),\alpha_{2}(g))}
\right)
\right]
\int_{g}\varphi, 
\end{split}
\]
where, for $g \in \fcal(P)_{m-1}$, the vector $\alpha_{g}$ is the inward primitive normal to $g$, 
and for $g \in \fcal(P)_{m-2}$, the vectors $\alpha_{1}(g),\alpha_{2}(g)$ are the inward primitive normal to 
the facets $g_{1},g_{2} \in \fcal(P)_{m-1}$ such that $g=g_{1} \cap g_{2}$. 
\end{cor}

\begin{proof}
By \eqref{AEMc}, the faces which contribute to $A_{2}(P;\varphi)$ is of $m-1$ or $m-2$ dimension. 
Let $g$ be a facet of $P$. Then, $D_{n}(P;g)$ is the lift of $D_{n}(\pi_{g}(C_{P}(g));0)$. 
Hence, as before, we have 
\[
D_{n}(P;g)=(-1)^{n-1}p(n)\nabla_{u_{g}}^{n-1}=-\frac{b_{n}}{n!}\nabla_{u_{g}}^{n-1} \ \ (n \geq 1),  
\]
where the rational vector $u_{g} \in L(g)^{\perp_{Q}}$ is given in \eqref{genPPP}. 
But, we are working in the standard Euclidean space with the integral lattice $\mb{Z}^{m}$ 
and the standard inner product. Since $P$ is Delzant, we can take an integral basis $e_{1},\ldots,e_{m}$ of $\mb{Z}^{m}$ 
such that $e_{1},\ldots,e_{m-1}$ is an integral basis of $L(g) \cap \mb{Z}^{m}$ and if we denote 
the dual basis of $e_{1},\ldots,e_{m}$ by $\alpha_{1},\ldots,\alpha_{m}$, then $\alpha_{m}=\alpha_{g}$. 
Thus, we have $u_{g}=\alpha_{g}/Q(\alpha_{g},\alpha_{g})$ and hence 
\[
D_{2}(P;g)=-\frac{b_{2}}{2!}\nabla_{\alpha_{g}/Q(\alpha_{g},\alpha_{g})}=-\frac{1}{12Q(\alpha_{g},\alpha_{g})}\nabla_{\alpha_{g}}. 
\]
Next, suppose that $g$ is a face of dimension $m-2$. Take two facets $g_{1}$, $g_{2}$ such that $g=g_{1} \cap g_{2}$. 
Denote $\alpha_{i}(g) \in \Lambda$ the primitive inward normal to $g_{i}$ $(i=1,2)$. 
Let $v$ be a vertex in $g$, and take $g_{3},\ldots,g_{m} \in \fcal(P)_{m-1}$ such that $\{v\}=g_{1} \cap \cdots \cap g_{m}$. 
Let $E=\{e_{1},\ldots,e_{m}\}$ be an integral basis of $\mb{Z}^{m}$ such that each vector $v+e_{j}$ defines 
an edge incident to $v$ and $v+e_{j} \not \in g_{j}$. We have 
\[
C_{P}(g)=\mb{R}_{+}e_{1}+\mb{R}_{+}e_{2}+L(g), 
\]
and $e_{3},\ldots,e_{m}$ is an integral basis of $L(g) \cap \mb{Z}^{m}$. 
Let $\alpha_{1},\ldots,\alpha_{m}$ be the dual basis of $e_{1},\ldots,e_{m}$. 
Then $\alpha_{i}=\alpha_{i}(g)$ for $i=1,2$, and $\alpha_{1},\alpha_{2}$ form a basis of $L(g)^{\perp}$. 
We write 
\[
e_{1}=u_{1}+v_{1},\ \ e_{2}=u_{2}+v_{2},\ \ u_{1},u_{2} \in L(g)^{\perp},\ \ v_{1},v_{2} \in L(g). 
\]
Under the identification 
\[
X/L(g) \ni x+L(g) \mapsto Q(x,\alpha_{1})u_{1}+Q(x,\alpha_{2})u_{2} \in L(g)^{\perp}, 
\]
the cone $\pi_{g}(C_{P}(g))$ is identified with $\mb{R}_{+}u_{1}+\mb{R}_{+}u_{2}$ and 
the generator of $\pi_{g}(\mb{Z}^{m})$ is identified with $u_{1},u_{2}$. 
Thus, by \eqref{2Dcone2}, we have 
\begin{equation}
\label{D2D2}
D_{2}(P;g)=D_{2}(\pi_{g}(C_{P}(g));0)=
\frac{1}{4}+\frac{1}{12}
\left(
\frac{Q(u_{1},u_{2})}{Q(u_{1},u_{1})}+
\frac{Q(u_{1},u_{2})}{Q(u_{2},u_{2})}
\right). 
\end{equation}
But then it is straight forward to show that 
\[
\begin{gathered}
Q(u_{1},u_{1})=\frac{Q(\alpha_{2},\alpha_{2})}{D},\ \ 
Q(u_{2},u_{2})=\frac{Q(\alpha_{1},\alpha_{1})}{D},\ \ 
Q(u_{1},u_{2})=-\frac{Q(\alpha_{1},\alpha_{2})}{D},\\ 
D=Q(u_{1},u_{1})Q(u_{2},u_{2})-Q(u_{1},u_{2})^{2}.  
\end{gathered}
\]
From this and \eqref{D2D2}, we conclude the assertion. 
\end{proof}

\bibliographystyle{amsalpha}

\end{document}